\documentclass[12pt]{amsart}
\usepackage{amsthm, amsrefs, amsmath, amssymb}
\usepackage[version-1-compatibility]{siunitx}
\usepackage{graphicx}
\usepackage{caption}
\usepackage{wrapfig}
\usepackage{tikz}
\usetikzlibrary{decorations.pathreplacing}
\usepackage{rotating}
\usepackage{diagbox}
\usepackage{lscape}

\newtheorem{theorem}{Theorem}[section]

\newtheorem{question}[theorem]{Question}
\newtheorem{example}[theorem]{Example}

\newtheorem{definition}[theorem]{Definition}
\newtheorem{corollary}[theorem]{Corollary}

\newtheorem{lemma}[theorem]{Lemma}
\newtheorem{notation}[theorem]{Notation}

\newtheorem{note}[theorem]{Note}
\numberwithin{equation}{section}

\graphicspath{ {./} }
\tikzset{every picture/.style={line width=0.75pt}}

\begin{document}

\setlength{\abovedisplayskip}{0cm}

\title{Rational Link Fertility}

\author{Andrew Ducharme}
\date{\today}

\maketitle

\begin{abstract}
A knot $K$ is the resultant of a knot $H$ if there exists a minimal crossing diagram $D$ of $K$ such that some crossings of $D$ can be altered to produce $H$. $K$ is fertile if every prime knot $H$ with crossing number less than $c(K)$ is a resultant of $K$. $K$ is $n$-fertile if every prime knot $H$ with crossing number less than $n$ is a resultant of $K$. We classify the fertility and fertility number of all rational links. This requires the introduction of the analogous concept of link fertility.
\end{abstract}

\section{Introduction}

Consider a minimal diagram of a knot $K$. Would the knot type be changed by switching the parity of a crossing from over to under or vice versa? How many other knots can be formed from a single knot? Can we form all smaller knots? The concept of knot fertility, introduced in \cite{MR3844207}, formalizes these questions.

\begin{definition}
A knot $K$ is \emph{fertile} if every prime knot $H$ with $c(H) < c(K)$ can be formed from the shadows of the minimal crossing diagrams of $K$.
\end{definition}

\begin{definition}
A knot $K$ has \emph{fertility number} $F(K) = m$, where all knots $H$ with $c(H) \leq m$ can be formed from the shadows of minimal crossing diagrams of $K$. We say $K$ is $p$-\emph{fertile} for all crossing numbers $p \leq m$.
\end{definition}

Instead of working with true knot diagrams, we start with \textit{shadows}, or projections, of the diagrams, where all over/undercrossing information is forgotten. We then examine the knots formed by particular choices of over/under parities for each crossing. The knot formed by a complete set of choices for each crossing is a \textit{resultant}. In Figure 1, for example, 1(a) is a diagram of the trefoil knot, 1(b) is the shadow of that diagram, and 1(c) is a particular resultant of that shadow.

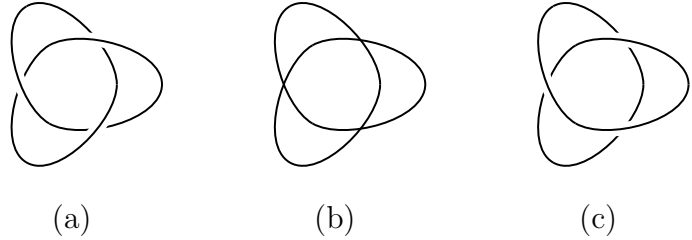
\begin{figure}[h]

\begin{center}
\begin{tikzpicture}

\begin{scope}[xshift = -3.5 cm, scale = 0.6]
	\coordinate (A) at (0:1cm) {};
	\coordinate (B) at (120:1cm) {};
	\coordinate (C) at (240:1cm) {};
	\coordinate (D) at (0:2cm) {};
	\coordinate (E) at (120:2cm) {};
	\coordinate (F) at (240:2cm) {};
	
	\draw (A) .. controls +(90:.8cm) and +(30:.8cm) .. (E);
	\draw[line width=5pt, white] (D)  .. controls +(90:.8cm) and +(30:.8cm) .. (B);
	\draw (C) .. controls +(-30:.8cm) and +(270:.8cm) .. (D)  .. controls +(90:.8cm) and +(30:.8cm) .. (B);
	\draw [line width=5pt, white] (F)  .. controls +(-30:.8cm) and +(-90:.8cm) .. (A);
	\draw (B) .. controls +(210:.8cm) and +(150:.8cm) .. (F)  .. controls +(-30:.8cm) and +(-90:.8cm) .. (A);
	\draw[line width=5pt, white] (E)  .. controls +(210:.8cm) and +(150:.8cm) .. (C);
	\draw (E)  .. controls +(210:.8cm) and +(150:.8cm) .. (C);
	
	\draw (0, -3) node {(a)};
\end{scope}

\begin{scope}[scale = 0.6]
	\coordinate (A) at (0:1cm) {};
	\coordinate (B) at (120:1cm) {};
	\coordinate (C) at (240:1cm) {};
	\coordinate (D) at (0:2cm) {};
	\coordinate (E) at (120:2cm) {};
	\coordinate (F) at (240:2cm) {};
	
	\draw (A) .. controls +(90:.8cm) and +(30:.8cm) .. (E);
	\draw (C) .. controls +(-30:.8cm) and +(270:.8cm) .. (D)  .. controls +(90:.8cm) and +(30:.8cm) .. (B);
	\draw (B) .. controls +(210:.8cm) and +(150:.8cm) .. (F)  .. controls +(-30:.8cm) and +(-90:.8cm) .. (A);
	\draw (E)  .. controls +(210:.8cm) and +(150:.8cm) .. (C);
	
	\draw (0, -3) node {(b)};
\end{scope}

\begin{scope}[xshift=3.5cm, scale = 0.6]
	\coordinate (A) at (0:1cm) {};
	\coordinate (B) at (120:1cm) {};
	\coordinate (C) at (240:1cm) {};
	\coordinate (D) at (0:2cm) {};
	\coordinate (E) at (120:2cm) {};
	\coordinate (F) at (240:2cm) {};
	
	\draw (A) .. controls +(90:.8cm) and +(30:.8cm) .. (E);
	\draw[line width=5pt, white] (D)  .. controls +(90:.8cm) and +(30:.8cm) .. (B);
	\draw (C) .. controls +(-30:.8cm) and +(270:.8cm) .. (D);
	\draw (D)  .. controls +(90:.8cm) and +(30:.8cm) .. (B);
	\draw (F)  .. controls +(-30:.8cm) and +(-90:.8cm) .. (A);
	\draw [line width=5pt, white] (C)  .. controls +(-30:.8cm) and +(-90:.8cm) .. (D);
	\draw (C) .. controls +(-30:.8cm) and +(270:.8cm) .. (D);
	\draw (B) .. controls +(210:.8cm) and +(150:.8cm) .. (F);
	\draw[line width=5pt, white] (E)  .. controls +(210:.8cm) and +(150:.8cm) .. (C);
	\draw (E)  .. controls +(210:.8cm) and +(150:.8cm) .. (C);
	
	\draw (0, -3) node {(c)};
\end{scope}

\end{tikzpicture}
\end{center}
\caption{(a) The trefoil knot, (b) the shadow of a trefoil, and (c) a resultant isotopic to the unknot.}
\end{figure}

The initial study of knot fertility calculated the fertility numbers of all knots with 10 or fewer crossings, discovering $3_1$, $4_1$, $5_2$, $6_2$, $6_3$, and $7_6$ are the only such fertile knots \cite{MR3844207}. Recent work showed there are no fertile alternating knots with greater than seven crossings \cite{MR4193872} while another duology \cite{MR4323911, kf2} determined any other extant fertile knot $K$ to have an odd crossing number $11 \leq c(K) \leq 21$, be genus 2, and have braid index 4.

In this paper, we completely determine the fertility and fertility number of all rational knots and links. This partially answers Question 2 posed in \cite{MR3844207}. In doing so, we must (trivially) extend the definition of fertility to links.

\begin{definition}
A link $L$ has $\mu (L)$ components. This is its \emph{component number}.
\end{definition}

\begin{definition}
A link $L$ is \emph{fertile} if every prime link $M$ with $c(M) < c(L)$ and $\mu(M) = \mu(L)$ can be formed from the shadows of the minimal crossing diagrams of $L$. $L$ has \emph{fertility number} $F(L) = n$, where all M with $c(M) \leq n$ and $\mu(M) = \mu(L)$ are resultants of $L$.
\end{definition}

\begin{theorem} (Corollary \ref{rat_fert})
The only fertile rational two-component links are $2^2_1$, $4^2_1$, $5^2_1$, $6^2_2$, $6^2_3$, and $7^2_2$.
\end{theorem}

Like with knots, we see the number of fertile links dropping rapidly as the number of link types abound with crossing number. Showing that these are the sole fertile rational links is actually very simple, a direct consequence of the maximum allowed number of unique resultants for each link. As for fertility number specifically, we explicitly determine every rational links' fertility number, finding a limiting value for both component lengths.


\begin{theorem} (Corollary \ref{link_max})
Rational two-component links \\ $L = N[a_1 \; a_2 \; ... \; a_n]$ have $F(L) = 6$ for $n > 3$.
\end{theorem}

\begin{theorem} (Theorem \ref{knot_max})
Rational knots $K = N[a_1 \; a_2 \; ... \; a_n]$ have $F(K) = 7$ for $n > 7$.
\end{theorem}

And while we do not touch on this idea in this paper, we suggest future study on the rational resultants of a link could be promising.

\begin{definition}
A knot $K$ has \emph{rational fertility number} $F_R(K) = n$, where all rational knots $H$ with $c(H) \leq n$ can be formed from the shadows of any minimal crossing diagram of $K$.
\end{definition}

\begin{question}
What is $F_R(L)$ for all rational links?
\end{question}

Since all knots with seven or fewer crossings are rational, $F_R(K) = F(K) \leq 7$. The smallest crossing number rational knots with $F_R(K) = 8$ are $11a77 = N[2 \; 2 \; 1 \; 1 \; 1 \; 1 \; 1 \; 2]$, $11a91 = N[2 \; 1 \; 1 \; 2 \; 1 \; 1 \; 1 \; 2]$, $11a96 = N[2 \; 2 \; 2 \; 1 \; 1 \; 1 \; 2]$, and $11a159 = N[2 \; 2 \; 2 \; 2\; 1 \; 2]$. We believe the techniques used in this paper could be extended to completely classify $F_R$ over the rational links. 

The paper is outlined as follows: Section \ref{link_fert} answers the binary question of fertility, Section \ref{link_fert_num} considers the fertility number of rational 2-component links, and Section \ref{knot_fert_num} considers the fertility number of rational knots. Whereas the rest of the paper considers the existence of various resultants, Section \ref{combin} counts the number of each for some rational links. Finally, in the Appendix, we tabulate the fertility number of every rational link.

\section{Rational Link Fertility} \label{link_fert}

\begin{theorem}
The only fertile rational knots are $3_1$, $4_1$, $5_2$, $6_2$, $6_3$, and $7_6$. The only fertile rational links are $2^2_1$, $4^2_1$, $5^2_1$, $6^2_2$, $6^2_3$, and $7^2_2$.
\end{theorem}

In this section, we'll define rational tangles and links, establish our procedure for determining their resultants, then prove the above statement about the binary invariant fertility.

\begin{definition}
A \emph{rational tangle} $T$ is a multiplication of integer tangles, defined pictorially below for two (left) and three or more (right) tangles, and is denoted by the list of such composing tangles $T = a_1 \; a_2 \; ... \; a_n$, $a_i \in \mathbb{Z}$. $T$ has length $d(T) = n$. The integer tangle $a_i$ has $|a_i|$ crossings.
\end{definition}

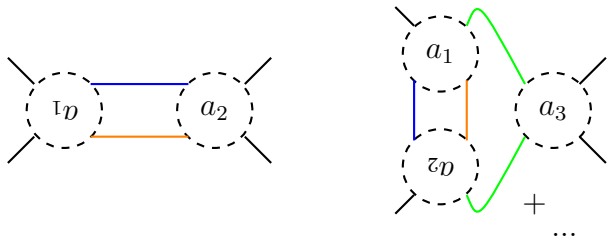
\begin{center}
\begin{tikzpicture}
\begin{scope}[xshift = -4 cm]


\draw [thick, dashed] (-1, 0) circle (0.5);
\draw (-1, 0) node {\scalebox{-1}[-1]{$a_1$}};

\draw [thick, dashed] (1, 0) circle (0.5);
\draw (1, 0) node {$a_2$};

\draw [thick, blue] (-0.65, 0.35) -- (0.65, 0.35);
\draw [thick, orange] (-0.65, -0.35) -- (0.65, -0.35);

\draw [thick] (-1.4, 0.35) -- (-1.75, 0.7);
\draw [thick] (-1.4, -0.35) -- (-1.75, -0.7);
\draw [thick] (1.4, 0.35) -- (1.75, 0.7);
\draw [thick] (1.4, -0.35) -- (1.75, -0.7);

\end{scope}


\draw [thick, dashed] (0, 0.75) circle (0.5);
\draw (0, 0.75) node {$a_1$};

\draw [thick, dashed] (0, -0.75) circle (0.5);
\draw (0, -0.75) node {\scalebox{-1}[-1]{$a_2$}};

\draw [thick, dashed] (1.5, 0) circle (0.5);
\draw (1.5, 0) node {$a_3$};

\draw [thick, blue] (-0.35, 0.4) -- (-0.35, -0.4);
\draw [thick, orange] (0.35, 0.4) -- (0.35, -0.4);
\draw [thick, green] (0.35, -1.125) .. controls (0.5, -1.5) .. (1.125, -0.35);
\draw [thick, green] (0.35, 1.125) .. controls (0.5, 1.5) .. (1.125, 0.35);

\draw [thick] (1.85, 0.35) -- (2.2, 0.7);
\draw [thick] (1.85, -0.35) -- (2.2, -0.7);
\draw [thick] (-0.35, 1.125) -- (-0.6, 1.375);
\draw [thick] (-0.35, -1.125) -- (-0.6, -1.375);

\draw (1.25, -1.25) node {+};
\draw (1.65, -1.65) node {...};

\end{tikzpicture}
\end{center}

We have two motivations for studying rational tangles: they are completely classified by their continued fraction, and low crossing number links are predominantly closures of rational tangles. 

\begin{definition}
A \emph{continued fraction} $(a_1 \; a_2 \; ... \; a_n)$, $a_i \in \mathbb{R}$ is a real number

\begin{equation*}
a_1 + \frac{1}{a_2 + \frac{1}{... +\frac{1}{a_{n-1} + \frac{1}{a_n}}}}
\end{equation*}

A continued fraction is in \emph{canonical form} if $|a_1| > 1$, $a_i \not= 0$, $1 < i \leq n$, $sgn(a_1) = sgn(a_j)$, $j \leq n$, and $a_i \in \mathbb{Z}$. Canonical forms are unique to a continued fraction.
\end{definition}

\begin{theorem}\cite{MR0258014}
Two rational tangles are isotopic to each other if their continued fraction is equivalent.
\end{theorem}

Consequently, we have unique canonical forms of rational tangles. Note $0$, $\infty$, and $\pm1$ are the canonical tangles of their isotopy class that do not fit with the above definition. Unless otherwise noted, in all our later results, a given rational link is in canonical form.

\begin{definition}
A tangle $T$ has two closures which maintain the diagram's crossing number: the \emph{numerator} $N[T]$ and \emph{denominator} $D[T]$.
\end{definition}

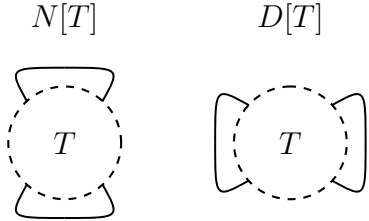
\begin{center}
\begin{tikzpicture}

\begin{scope}
\draw(0, 1) .. controls (-0.75, 1) .. (-0.5, 0.55);
\draw(0,1) .. controls (0.75, 1) .. (0.5, 0.55);

\draw [thick, dashed] (0, 0) circle (0.75);
\draw (0, 0) node {$T$};

\draw(0, -1) .. controls (-0.75, -1) .. (-0.5, -0.55);
\draw(0, -1) .. controls (0.75, -1) .. (0.5, -0.55);

\draw (0, 1.65) node {$ N[T] $};

\end{scope}

\begin{scope}[xshift = 3cm, rotate=90]
\draw(0, 1) .. controls (-0.75, 1) .. (-0.5, 0.55);
\draw(0,1) .. controls (0.75, 1) .. (0.5, 0.55);

\draw [thick, dashed] (0, 0) circle (0.75);
\draw (0, 0) node {$T$};

\draw(0, -1) .. controls (-0.75, -1) .. (-0.5, -0.55);
\draw(0, -1) .. controls (0.75, -1) .. (0.5, -0.55);

\draw (1.65, 0) node {$ D[T] $};
\end{scope}

\end{tikzpicture}
\end{center}

\noindent Note $N[T \; 0] \simeq D[T] \simeq N[0 \; T]$ and $N[T \; \infty] \simeq N[T]$. We can concern ourself solely with numerator closures because $D[a_1 \; a_2 \; ... \; a_n] \simeq N[a_1 \; a_2 \; ... \; a_{n-1}]$.

\begin{theorem}\label{KnotFracEquiv}
 \cite{MR82104, MR1953344} Suppose there exists two rational tangles with continued fractions $\frac{p}{q}$ and $\frac{p_0}{q_0}$, where p, q and $p_0, q_0$ are relatively prime. If N$\left[\frac{p}{q}\right]$ and N$\left[\frac{p_0}{q_0}\right]$ are the links formed by the closure of these tangles, then the links are equivalent (up to isotopy) iff:
\begin{itemize}
\item $p = p_0$
\item either q $\equiv$ $q_0$ mod p or $qq_0$ $\equiv$ 1 mod p.
\end{itemize}
\end{theorem}

\begin{corollary}\label{palTheorem} \cite{MR1953344}
$N[a_1 \; a_2 \; ... \; a_{n-1} \; a_n] \simeq N[a_n \; a_{n-1} \; ... \; a_2 \; a_1]$.
\end{corollary}

This forces $a_n > 1$ in reduced rational links, since the continued fraction must be in canonical form when written forwards and backwards.

\begin{notation}
The shadow of an integer tangle $a_i$, or a \emph{shadow tangle}, is denoted $a_i'$. 
\end{notation}

\begin{center}
\begin{tikzpicture}
\begin{scope}
\draw (-.5, -.5) -- (.5,.5);
\draw (-.5,.5) -- (.5,.-.5);
\end{scope}
\begin{scope}[xshift=1.25cm]
\draw (-.5, -.5) -- (.5,.5);
\draw (-.5,.5) -- (.5,.-.5);
\end{scope}
\begin{scope}[xshift=-1.25cm]
\draw (-.5, -.5) -- (.5,.5);
\draw (-.5,.5) -- (.5,.-.5);
\end{scope}
\begin{scope}[xshift=2.5cm]
\draw (-.5, -.5) -- (.5,.5);
\draw (-.5,.5) -- (.5,.-.5);
\end{scope}
\begin{scope}[xshift=-2.5cm]
\draw (-.5, -.5) -- (.5,.5);
\draw (-.5,.5) -- (.5,.-.5);
\end{scope}

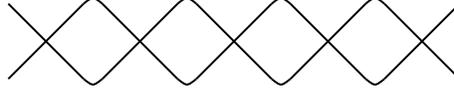
\draw (0,-1) node {The 5' tangle};
 
\draw (.5,.5) .. controls (.625, .6) .. (.75,.5);
\draw (.5, -.5) .. controls (.625, -.6) .. (.75, -.5);
\draw (-.5, .5) .. controls (-.625, .6) .. (-.75, .5);
\draw (-.5, -.5) .. controls (-.625, -.6) .. (-.75, -.5);
\draw (1.75, .5) .. controls (1.875, .6) .. (2, .5);
\draw (1.75, -.5) .. controls (1.875, -.6) .. (2, -.5);
\draw (-1.75, .5) .. controls (-1.875, .6) .. (-2, .5);
\draw (-1.75, -.5) .. controls (-1.875, -.6) .. (-2, -.5);
\end{tikzpicture}
\end{center}

\begin{theorem}
Let $T'$ be the shadow of an integer tangle with $t$ crossings. If $t$ is odd, $T'$ has resultants $t$, $t-2$, $t-4$, ..., $+1$, $-1$, $-3$, ..., $-t$.
If $t$ is even, $T'$ has resultants $t$, $t-2$, $t-4$, ..., $+2$, $0$, $-2$, ..., $-t$.
\end{theorem}

\begin{proof}
We can label crossings as either positive or negative by the sign of the overstrand's slope. Given a shadow tangle, we choose some to be positive and the rest negative. Neighboring positive and negative crossings are removed via RII, decreasing the crossing number by two, until only crossings of one type remain.
\end{proof}

Because having a link $L$ as a resultant implies the mirror image of $L$ is also a resultant, we can consider opposite sign tangles $+t$ and $-t$ as one case $\pm t$ for all $t \in \mathbb{N}$. For multiple integer tangles, we need to have all the same parity $\pm$ or $\mp$ to ensure the overall tangle is in canonical form. For example, $\pm 3 \; \pm 4 \; \pm 2$ is in canonical form and completely simplified. 

The next result, fundamental to this work, determines how to handle tangles with different parities like $\pm 3 \; \mp 4 \; \pm 2$.

\begin{lemma} \label{untangle}
(The Untangling Lemma) Let $\pm a_i \; \mp a_{i+1}$, $a_i \geq 1$, be an assignment of two adjacent integer tangles inside a rational link shadow $L' = N[a_1' \; a_2' \; ... \; a_n']$, $1 < i < n$. Then 

 \[
    	\pm a_i \; \mp a_{i+1} \approx \begin{cases}
 	\pm (a_i-1) \; \pm1 \; \pm(a_{i+1}-1), & a_{i+1} \geq 2 \\
        \pm (a_i-1) \; a_{i+2}' \pm 1, & a_{i+1} = 1 \\
        \pm a_i + a_{i+2}', & a_{i+1} = 0, 
        \end{cases}. 
  \]
  
\noindent For $a_{i+1} > 0$, this move reduces the diagram's crossing number by one.
\end{lemma}

\begin{proof}
Assign the first $j$ tangles $a_i' \rightarrow \pm c_i$ and $a_{j+1} \rightarrow \mp c_{j+1}$ so $c_j \; c_{j+1}$ is the leftmost case of disagreeing parities and $2 \leq c_j \leq a_j$. Then

\begin{equation*}
L' \rightarrow N[\pm c_1 \; \pm c_2 \; ... \; \pm c_j \; \mp c_{j+1} \; a_{j+2}' \; ... \; a_n'].
\end{equation*}

Let $A \in \mathbb{R}$ be the continued fraction of the rational tangle composed of the $j + 2$th through $n$th integer tangles once assigned. Then the rational tangle composed of the $j$th through $n$th integer tangles has a continued fraction 

\begin{equation*}
c_j + \frac{1}{-c_{j+1} + \frac{1}{A}} = c_j - 1 + \frac{-(c_{j+1} - 1) + \frac{1}{A}}{-c_{j+1} + \frac{1}{A}}
= (c_j - 1) + \frac{1}{1 + \frac{1}{(c_j - 1) - \frac{1}{A}}}.
\end{equation*} 

For a continued fraction $x = (a, b, ..., n)$, $-x = (-a, -b, ..., -n)$ so the negation of $A$ flips the sign of all its components. Translating the above result into a partial shadow of $L'$, we get 

\begin{equation*}
N[\pm c_1 \; \pm c_2 \; ... \; \pm (c_j-1) \; \pm 1 \; \pm (c_{j+1} - 1) \; -a_{j+2}' \; ... \; -a_n']
\end{equation*}

\noindent However, a shadow tangle does not change upon negation—it already contains the positive and negative assignments.

The second and third cases are consequences of $k \; 0 \; \ell \simeq (k+\ell)$ for generic integer tangles $k$, $\ell$.
\end{proof}

The key behind this lemma is $N[\pm c_1 \; ... \; \pm c_j \; a_{j+1}' \; ... \; a_n']$ as a notational shorthand for the set of rational links formed from all assignments of the remaining shadow tangles. The rational link $N[\pm c_1 \; \mp c_2 \; \pm c_3 \; ... \; \pm c_n]$ is isotopic to $N[\pm (c_1 - 1) \; \pm 1 \; \pm (c_2 - 1) \; \mp c_3 \; ... \; \mp c_n]$ but the set $N[\pm c_1 \; \mp c_2 \; a_3' \; ... \; a_n']$ is identical to the set $N[\pm (c_1 - 1) \; \pm 1 \; \pm (c_2 - 1) \; a_3' \; ... \; a_n']$. We can save ourselves a lot of accounting for negative signs by ``absorbing" the negative sign into the shadow tangles and treating the remaining tangles as untouched.

Technically, on an individual link level, we're choosing a particular assignment for the first part of a link, applying the isotopy of the Untangling Lemma, then going to the set of remaining possible assignments and trading in for another link entirely. As such, the link we start with is, in general, not isotopic to the link we end with. However, this trade is bijective; the entire resolution set is unchanged. 

\begin{center}
\begin{tikzpicture}
\draw node (orig) at (-6, 2.5) {$N[\pm a_1 \; \mp a_2 \; \pm a_3]$};

\draw node (final) at (0, 2.5) {$N[\pm (a_1-1) \; \pm 1 \; \pm (a_2-1) \; \pm a_3]$};

\draw node (isotope) at (-6, 0) {$N[\pm (a_1-1) \; \pm 1 \; \pm (a_2-1) \; \mp a_3]$};

\draw [thick, <->] (orig) -- (final) node[midway,above] {$\approx$};

\draw [thick, <->] (orig) -- (isotope) node[midway, left] {$\simeq$};

\node[rectangle, draw, fill=black, text=white, outer sep = 5pt] (r) at (0,1.0625) {\large $\{N[\pm a_1 \; \mp a_2 \; a_3']\}$};

\draw [thick, dashed, ->] (r.north) -- (final);
\draw [thick, dashed, ->] (isotope.east) .. controls (0, 0) .. (r.south);

\end{tikzpicture}
\end{center}

\noindent To remind ourselves that something more complicated is going on under the hood, we use the symbol $\approx$ for equivalence instead of the standard isotopy symbol $\simeq$.

\begin{theorem}
A rational link $L = N[\pm a_1 \; \pm a_2 \; ... \; \pm a_n]$ has only one resultant $R$ with $c(R) = c(L)$ and at most $n-1$ unique resultants with $c(R) = c(L) - 1$. The same crossing number resultant is $L$.
\end{theorem}

\begin{proof}
Changing a single crossing inside an integer tangle will decrease that tangle's crossing number by two. Changing more crossings will decrease the crossing number by two until half the crossings have been changed, at which point the crossing number will start to increase by two, until which every crossing has been changed. Therefore, the only resultant diagrams $D$ with $c(D) = c(L)$ are those where no crossings are changed, or all the crossings in some tangles $a_i$ have been changed. 

The former diagram is $L$. The latter diagrams are not in canonical form, but can isotoped so by the above lemma. For each integer tangle with all crossings changed, the resultant crossing number decreases by 1, so the only route to resultants with $c(R)=c(L)-1$ is to alter a single tangle. There are ${n \choose 1} = n$ such diagrams, however, we already account for the mirroring of the first tangle by the $\pm$ notation as $N[\pm a_1 \; a_2' \; ... \; a_n']$ has the both the initial assignment ($+$) and its mirror ($-$).
\end{proof}

This result alone(!) explains why knots $5_1 = N[5]$, $6_1 = N[2 \; 4]$, $7_1 = N[7]$, $7_3 = N[3 \; 4]$, $7_4 = N[3 \; 1 \; 3]$, and $7_5 = N[3 \; 2 \; 2]$, and two-component links $6^2_1 = N[6]$, $7^2_1 = N[4 \; 1 \; 3]$, and $7^2_3 = N[2 \; 3 \; 2]$ are infertile. They are all too short to produce the variety of links with one fewer crossing!

\begin{table}[h]
\begin{tabular}{cc|cc}
$c(K)$ & \# $K$ & $c(L^2)$ & \# $L^2$ \\ \hline
0    & 1    & 0                       & 1                       \\
3    & 1    & 2                       & 1                       \\
4    & 1    & 4                       & 1                       \\
5    & 2    & 5                       & 1                       \\
6    & 3    & 6                       & 3                       \\
7    & 7    & 7                       & 8                       \\
8    & 21   & 8                       & 16                      \\
9    & 49   & 9                       & 61                     
\end{tabular}
\caption{The number of one- and two-component links $K$ and $L^2$ for particular crossing numbers.}
\end{table}

\begin{theorem}
$d(L) \leq c(L) - 2$. 
\end{theorem}

\begin{proof}
We can construct the longest possible link with $n$ crossings by multiplying $n$ single crossing tangles. Two neighboring $\pm 1$ tangles on the interior of a rational link only have one direct connection and remain distinct. However, there are two direct $a_1$-$a_2$ and $a_{n-1}$-$a_n$ connections in the numerator closure, so the $a_1$ and $a_2$ tangles, and the $a_{n-1}$ and $a_n$ tangles each cohere into integer tangles with two crossings, decreasing link length by two.
\end{proof}

\begin{corollary} \label{rat_fert}
A rational link $L$ is infertile if $c(L) \geq 8$. Specifically, $F(L) \leq c(L) - 2$.
\end{corollary}

\begin{proof}
$L$ has less than or equal to $c(L) - 2$ resultants $R$ with $c(R) = c(L) - 1$. There are 7 knots and 8 two component links with seven crossings, but eight crossing links can form at most 5 resultants with that crossing number. The number of rational links grows much faster than crossing number.
\end{proof}

With this, we've proven all rational links are infertile except the fertile knots $3_1$, $4_1$, $5_2$, $6_2$, $6_3$, and $7_6$, the fertile links $2^2_1$, $4^2_1$, $5^2_1$, $6^2_2$, $6^2_3$, and $7^2_2$, and the pesky $7_7 = N[2 \; 1 \; 1 \; 1 \; 2]$. This length based approach puts a rigorous spin on the idea that low crossing number knots are fertile because their resultants are pidgeonholed into the same knots. At small crossing numbers, a reduced resultant can only be one of very few options. But as crossing number increases, unlocking more diverse knots, resultants of a particular crossing number are no longer so constrained.

For completeness, we'll give a lack-of-resultants proof here for why $7_7$ is infertile, but we'll have more to say about knots like $N[2 \; 1 \; 1 \; 1 \; 2]$ later.

\begin{theorem}
If a rational link $L$'s list $\{a_i\}$ is symmetric \\
$N[a_1 \; a_2 \; ... \; a_{n-1} \; a_n] = N[a_1 \; a_2 \; ... \; a_2 \; a_1]$, it has $\lfloor \frac{n}{2} \rfloor$ resultants $R$ with $c(R) = c(L) - 1$. 
\end{theorem}

\begin{proof}
Let $d(L) = n$ so $k = \lceil \frac{n}{2} \rceil$ and $a_k$ is the fulcrum of $L$'s reversal symmetry. As we're only changing one tangle, the number of unique resultants is the number of unique pairs $(1,2), (2,3), (3,4), ..., (k-1, k)$. If $n$ is odd, this is $\frac{n-1}{2}$. If $n$ is even, we have the additional pair $(k,k)$ and $\frac{n}{2}$ unique pairs, hence the floor function.
\end{proof}

\begin{corollary}
$7_7 = N[2 \; 1 \; 1 \; 1 \; 2]$ is infertile.
\end{corollary}

\begin{proof}
$\lfloor \frac{d(7_7)}{2} \rfloor = 2$. There are three prime knots with six crossings.
\end{proof}

\section{Rational Link Fertility Number} \label{link_fert_num}

\begin{theorem}
A rational knot $K$ has $F(K) \leq 7$ and a two component rational link $L$ has $F(L) \leq 6$.
\end{theorem}

\begin{proof}
A rational link's resultants must be rational. Both bounds are the largest crossing number with only rational links for both component numbers.
\end{proof}

\begin{definition}
A rational link $L$ is \emph{locally fertile} if it is a knot and $F(L) = 7$, or it is a two-component link and $F(L) = 6$.
\end{definition}

We'd expect once rational links are large enough, they should be locally fertile. But what do we mean by ``large enough?" It's reasonable to anticipate fertility number growing roughly with crossing number (the first locally fertile knot is $9_{27}$ and the next smallest are $10_{23}, 10_{25}, 10_{26}$ and $10_{39}$–$10_{45}$, barring $10_{43}$) but as we've seen so far, it is described more accurately by length. Our goal in this and the next section is to build and use tools to determine the minimal rational link length for local fertility.

\begin{theorem}
$F(N[a_1 \; ... \; a_{i-1} \; (a_i + 2) \; a_{i+1} \; ... \; a_n]) \geq$\\
$F(N[a_1 \; ... \; a_{i-1} \; a_i \; a_{i+1} \; ... \; a_n])$, $1 \leq i \leq n$. 
\end{theorem}

\begin{proof}
Call $N[a_1 \; ... \; a_{i-1} \; a_i \; a_{i+1} \; ... \; a_n]$ $L_1$ and $N[a_1 \; ... \; a_{i-1} \; (a_i + 2) \; a_{i+1} \; ... \; a_n]$ $L_2$. The shadow of $L_2$ can be resolved into three mutually distinct sets: $N[a_1' \; ... \; a_{i-1}' \; a_i' \; a_{i+1}' \; ... \; a_n']$, $N[a_1' \; ... \; a_{i-1}' \; \pm (a_i + 2) \; a_{i+1}' \; ... \; a_n']$, and $N[a_1' \; ... \; a_{i-1}' \; \mp (a_i + 2) \; a_{i+1}' \; ... \; a_n']$. The first set is simply all the resultants of $L_1$, thus $F(L_2) \geq F(L_1)$.
\end{proof}

We can use this heredity to build the least complicated rational links at every length and keep adding crossings to increase our lower bounds on fertility number. We will call this set of least complicated links the \textit{trunk}. A more technical definition is:

\begin{definition}
The \emph{trunk} $\mathcal{K}_\ell$ is the set of rational links $L = N[a_1 \; a_2 \; ... \; a_\ell] = N[\{a_i\}]$ where $(a_1, a_\ell) = (2,2), (3,2)$, or $(3,3)$, and $\{a_j\}$, $1 < j < \ell$ are $(\ell - 2)$-tuples with elements $\{1, 2\}$.
\end{definition}

\begin{example}
$\mathcal{K}_1 = \{N[2], N[3]\}$, $\mathcal{K}_2 = \{N[2 \; 2], N[2 \; 3], N[3 \; 3]\}$, and
$\mathcal{K}_3 = \{N[2 \; 1 \; 2], N[2 \; 2 \; 2], N[3 \; 1 \;  2], N[3 \; 2 \; 2], N[3 \; 1 \; 3], N[3 \; 2 \; 3]\}$
\end{example}

For any link $L \in \mathcal{K}_\ell$, $\ell + 2 \leq c(L) \leq 2\ell + 2$. $|\mathcal{K}_\ell | \leq 3 \times 2^{\ell - 1}$, but the bound will not be achieved for $\ell > 3$. If we included a link for each $(\ell - 2)$-tuple, nonsymmetric tuples would duplicate its reverse if $a_1 = a_n$, e.g. the 2-tuple $(1,2)$ is included in six links $N[2 \; 1 \; 2 \; 2] = N[2 \; 2 \; 1 \; 2]$, $N[3 \; 1 \; 2 \; 2]$, $N[3 \; 2 \; 1 \; 2]$, and $N[3 \; 1 \; 2 \; 3] = N[3 \; 2 \; 1 \; 3]$.

\begin{definition}
Every rational link $L_2$ of length $\ell$ can be written as $N[(a_1 + 2m_1) \; (a_2 + 2m_2) \; ... \; (a_\ell + 2m_\ell)]$, where $L_1 = N[a_1 \; a_2 \; ... \; a_\ell] \in \mathcal{K}_\ell$. We call $L_2$ a \emph{branch} of $L_1$. 
\end{definition}

\begin{definition}
$g(l, m)$ is the smallest fertility number of any m-component link in a particular trunk $\mathcal{K}_{l}$.
\end{definition}


\begin{corollary}
Let $L$ be a rational link. Then $F(L) \geq g(d(L), \mu(L))$.
\end{corollary}

This is why trunks are meaningful: they are not just the least complicated, but the \textit{least fertile} rational links for a given length. 

\begin{theorem} \label{f_across_L}
Let $L = N[a_1 \; a_2 \; ... \; a_n]$. Then
 \[
    	F(L) \geq \begin{cases}
 	F(N[a_1 \; a_2 \; ... \; a_{n-2}]), & a_{n} \; even \\
        F(N[a_1 \; a_2 \; ... \; (a_{n-1}+1)]), & a_n \; odd \\
        \end{cases}.
  \]
\end{theorem}

\begin{proof}
If $a_n$ even, $a_n \rightarrow 0 \implies L \rightarrow N[a_1 \; a_2 \; ... \; a_{n-2}]$, so all the resultants of the shorter link are the resultants of $L$. If $a_n$ odd, assignments $L \rightarrow N[a_1' \; a_2' \; ... \; a_{n-1}' \; \pm 1]$ and $N[a_1' \; a_2' \; ... \; a_{n-1}' \; \mp 1)]$ form all the resultants of $N[a_1' \; a_2' \; ... \; (a_{n -1})' \; 1']$  $\approx N[a_1' \; a_2' \; ... \; (a_{n -1} + 1)']$.
\end{proof}

\begin{notation}
$N[x \; 1^{\ell} \; y] \equiv N[x \underbrace{1 \; 1 \; ... \; 1}_{\ell \text{ unit tangles}} y]$
\end{notation}


\begin{theorem} \label{end_goal}
If $N[2 \; 1^{n-2} \; 2]$ and all rational links $L_B$ of $\ell$ components in $\mathcal{K}_{n-1}$ and $\mathcal{K}_{n-2}$ are $k$-fertile, all rational links $L$ of length $d(L) \geq n$ are $k$-fertile.
\end{theorem}

\begin{proof}
Let $L = N[a_1 \; a_2 \; ... \; a_n]$ be a rational link of $\ell$ components. If $a_1$ or $a_n$ are odd, $F(L) \geq F(N[a_1 \; a_2 \; ... \; (a_{n-1}+1)]) \geq g(n-1, \ell)$. Similarly, if $a_1, a_n$, and some tangles $a_j$, $1 < j < n$, are even, $F(L) \geq g(n-2, \ell)$. This leaves one pathological counterexample: if all $a_j$, $1 < j < n$, are odd, $a_1$ and $a_n$ are even, and $a_3 = a_{n-2} = 1$, assigning an end tangle to zero forms $M = N[a_1 \; ... \; a_{n-3} \; 1] \approx N[a_1 \; ... \; (a_{n-3} + 1)]$, so $d(M) = n-3$. Hence, $M$ is a branch of $N[2 \; 1^{n-5} \; 2] \in \mathcal{K}_{n-3}$.

By definition, $L$ in this case is a branch of $N[2 \; 1^{n-2} \; 2]$, which by having more crossings and being of longer length, is likely to have a larger fertility number than $N[2 \; 1^{n-5} \; 2]$. So while it is sufficient to require all links in $\mathcal{K}_{n-1}$, $\mathcal{K}_{n-2}$, and $\mathcal{K}_{n-3}$ to be $k$-fertile, simply using $\mathcal{K}_{n-1}$, $\mathcal{K}_{n-2}$, and $N[2 \; 1^{n-2} \; 2] \in \mathcal{K}_{n}$ is a much stronger condition.
\end{proof}

Our next step is to find $g(\ell, 1)$ and $g(\ell, 2)$ for increasingly large $\ell$ to determine the minimum fertility number at every length. Once these minima reach the local fertility number, we will have squeezed out the length where all longer links are locally fertile.

%
%
%
%
%
%
%

\subsection{Length 1 and 2 Knots and Links}

\begin{theorem}
$N[a_1]$ links, or the $T(p,2)$ torus links, have resultants $N[\pm k]$, where $a_1$ and $k$ have the same parity. 
\end{theorem}

\begin{proof}
This follows directly from the resolution of integer tangles.
\end{proof}

\begin{corollary}
$F(T(p, 2)) = 3$, $p > 1$, odd. The trefoil is the only fertile $T(p,2)$ knot.
\end{corollary}

\begin{proof}
These knots cannot have the figure-eight knot $N[2 \; 2]$ as a resultant.
\end{proof}

This has been repeatedly discovered \cite{MR3844207, MR3084750, MR1001742, MR4190429, MR3356086, MR2514545}.

\begin{corollary}
The Hopf link $2^2_1 = N[2]$ and $4^2_1 = N[4]$ are fertile ($F(N[2]) = 2$, $F(N[4]) = 4$). Otherwise, $N[a_1]$ links, $a_1 > 4$, have fertility number $4$.
\end{corollary}

\begin{proof}
These links cannot form $5^2_1 = N[2 \; 1 \; 2]$.
\end{proof}

\begin{theorem} \label{2n_closed}
All resultants of $N[2' \; a_2']$ shadows are $N[2 \; \ell]$ knots, $0 \leq \ell \leq a_2$.
\end{theorem}

This too is well-trodden \cite{MR3844207, MR3084750, MR4190429, MR3356086}. 

\begin{theorem}
Let $a_1, a_2, k, \ell$ be even, positive integers, $2 \leq k \leq a_1$, $2 \leq \ell \leq a_2$. Then the $N[a_1' \; a_2']$ knot shadow has the $N[k \; \ell]$ and $N[ (k-1) \; 1 \; (\ell-1)]$ knots and the unknot as resultants.
\end{theorem}

\begin{note}
$N[\pm (2 - 1) \; \pm 1 \; \pm (2 - 1)] \approx N[\pm 3]$.
\end{note}

\begin{proof}
We can automatically untangle the knot shadow for assignments $a_1'$ or $a_2' \rightarrow 0$. We now must identify $N[ m \; n ]$ diagrams, where $m$, $n \not= 0$. For specific positive indices $k$, $\ell$, we have knot diagrams $N[ \pm k \; \pm \ell]$ and $N[ \pm k \; \mp \ell] \simeq N[\pm (k - 1) \; \pm 1 \; \pm (\ell - 1)]$. Diagrams with matching signs are alternating, thus fully reduced. 
\end{proof}

\begin{corollary}
The nontrivial knots $N[a_1 \; a_2]$, $a_1$, $a_2$ even, have $F(N[a_1 \; a_2]) = 4$. The only such fertile knot is the figure-eight knot $4_1 = N[2 \; 2]$.
\end{corollary}

\begin{proof}
The choice $a_1' \rightarrow \pm 2$ and $a_2' \rightarrow \mp 2$ produces the trefoil. Any other nontrivial resultant $R$ has $d(R) \geq 2$, thus $5_1 = N[5]$ cannot be a resultant.
\end{proof}

\begin{theorem}
Let $a_1, k$ be even integers and $a_2, \ell$ be odd integers, $2 \leq k \leq a_1$, $3 \leq \ell \leq a_2$. Then the $N[a_1' \; a_2']$ knot shadow has resultants of the forms $N[k+1]$, $N[k \; \ell]$, and $N[(k-1) \; 1 \; (\ell-1)]$, in addition to the unknot.
\end{theorem}

\begin{proof}
Without loss of generality, choose $a_1$ to be even. Like above, we can take $a_1 \rightarrow 0$ to produce unknots, and can otherwise form $N[\pm k \; \pm \ell]$ and $N[\pm k \; \mp \ell] \approx N[\pm (k-1) \; \pm 1 \; \pm (\ell-1)]$ diagrams.
 
This leaves $N[\pm k \; \pm1] \simeq N[\pm(k+1)]$ and $N[\pm k \; \mp 1] \simeq N[\pm(k - 1)]$. $N[\pm 2 \; \mp 1] \simeq N[\pm 1 \; \pm 1 \; 0] \simeq D[\pm 2],$ or the unknot. 
\end{proof}

Theorem \ref{2n_closed} follows from $k = 2 \implies N[(k-1) \; 1 \; (\ell-1)] = N[1 \; 1 \; (\ell-1)] \simeq N[2 \; (\ell-1)]$.

\begin{corollary}
The only fertile knots of the form $N[2 \; a_2]$, $a_2 \in \mathbb{N}$, are $N[2 \; 3] = 5_2$, $N[2 \; 2] = 4_1$, $N[2 \; 1] = 3_1$, and $N[2 \; 0] = 0_1$. If $a_2 \geq 2$, $F(N[2 \; a_2]) = 4$.
\end{corollary}

\begin{proof}
$5_1 = N[5]$ cannot be written as a $N[2 \; \ell]$ knot: $\ell + \frac{1}{2} \not= 5$ for any $\ell \in \mathbb{Z}$.
\end{proof}

\begin{corollary}
Nontrivial knots $N[a_1 \; a_2]$, $a_1$ even, $a_2$ odd, $4 \leq a_1$, $2 < a_2$ are infertile with $F(N[a_1 \geq 4 \; \; a_2]) = 5$.
\end{corollary}

\begin{proof}
$a_2$ odd allows the formation of $N[5] = 5_1$, however, all resultants $R$ have $d(R) \leq 3$, precluding formation of $6_3 = N[2 \; 1 \; 1 \; 2]$.
\end{proof}


For two-component rational links $L$, we already know if $d(L) > 1$, $F(L) \geq 5$. For further information, we just need to determine if a rational link can form $6^2_1 = N[6]$, $6^2_2 = N[3 \; 3]$, and $6^2_3 = N[2 \; 2 \; 2]$.

\begin{theorem}
Let $L = N[a_1 \; a_2]$ be a rational two-component link. Then $F(L) = 5$.
\end{theorem}

\begin{proof}
The longest resultants of any $L$ are of the form $N[\pm c_1 \; \pm 1 \; \pm c_3] \implies 6^2_3 = N[2 \; 2 \; 2]$ cannot be a resultant.
\end{proof}

\subsection{Length 3 and Higher Links}

\begin{theorem}
Let $L$ be a nontrivial prime two-component link. If $L = 2_1^2  = T(2,2)$, $F(L) = 2$. If $L = T(p,2)$, $p$ even, $p \geq 4$, $F(L) = 4$. Otherwise, $F(L) \geq 5$.
\end{theorem}

\begin{proof}
In \cite{MR1030506}, Taniyama proved any prime two-component link of five or more crossings majorizes the links $2_1^2$, $4_1^2$, and $5_1^2$. \cite{MR3844207} showed a link that majorizes some link $L$ has $L$ as a resultant.
\end{proof}

\begin{theorem} \label{3tangle}
Let $L = N[a_1 \; a_2 \; a_3]$ be a nontrivial two-component link. If $a_2 = 1$ or $a_1 = a_3 = 2$, $F(L) = 5$. Otherwise $F(L) = 6$.
\end{theorem}

The rational link $N[a_1 \; a_2 \; a_3]$ has two components when $a_1, a_3$ even, $a_2$ odd, or $a_1, a_3$ odd, $a_2$ even, or $a_1, a_2, a_3$ even.

\begin{lemma}
Let $L = N[a_1 \; a_2 \; a_3]$ be a nontrivial two-component link. $6^2_3$ is a resultant iff $a_2 > 1$.
\end{lemma}

\begin{proof}
Suppose $a_2 = 1$. Flipping the sign of one or more integer tangles will produce resultants $R$ of the forms $N[\pm c_1 \; \pm 1 \; \mp c_3]$ or $N[\pm c_1 \; \mp 1 \; \pm c_3] \implies d(R) \leq 2$. Therefore, the only length three resultants will be $N[\pm c_1 \; \pm 1 \; \pm c_3]$, thus $6^2_3$ is not a resultant.

We now prove if $a_2 > 1$, $6^2_3$ is a resultant of $L$. If $a_2$ is even, a direct assignment $N[\pm2 \; \pm2 \; \pm2] = 6^2_3$ is possible. If $a_2 > 1$ is odd, the assignment $N[\pm 2 \; \mp 3 \; \pm 2]$ is possible. In the remaining case, $a_1, a_3$ odd and $a_2$ even, the assignment $N[\pm3 \; \mp 2 \; \mp 3]$ is possible.
\end{proof}

\begin{lemma}
Let $L = N[a_1 \; a_2 \; a_3]$ be a nontrivial two-component link. $6^2_1$ is a resultant iff one of $a_1, a_3 \not=2$.
\end{lemma}

\begin{proof}
Let $a_1 = a_3 = 2$. $a_1'$ or $a_3' \rightarrow 0$ produce only the un- or Hopf links, so we must consider $N[\pm 2 \; a_2' \; \pm 2]$ and $N[\pm 2 \; a_2' \; \mp 2]$. The only length one resultant of $N[\pm 2 \; a_2' \; \pm 2]$ for either parity of $a_2$ is $4^2_1$. 

For $N[\pm 2 \; a_2' \; \mp 2] $, $a_2' \rightarrow 0$ forms the unlink. Otherwise, we have $N[\pm 2 \; \pm a_2 \; \mp 2]$ and $N[\pm 2 \; \mp a_2 \; \mp 2] \approx N[\pm 2 \; \pm (a_2 - 1) \; \mp 2]$. Since $a_2$ can be either parity, we are essentially considering the same case $R = N[\pm 2 \; \pm (c_2 - 2) \; \pm 2]$. If $c_2 < 3$, $c(R) < 6$. If $c_2 \geq 3$, $d(R) = 3$. Therefore, $6^2_1$ cannot be a resultant.

We now show if one of $a_1, a_3 > 2$, then we can form $6^2_1$. If $a_2$ odd, we can form $N[\pm 4 \; \pm 1 \; \mp 2] \approx N[\pm 6]$. If $a_2$ even, we can form $N[\pm 4 \; 0 \; \pm 2]$ or $N[\pm 3 \; 0 \; \pm 3]$.
\end{proof}

We now prove Theorem \ref{3tangle}, completing the classification of length 3 rational links.

\begin{proof}
By the above lemmas, if $a_2 = 1$ or $a_1 = a_3 = 2$, $F(L) = 5$. We now check if larger links can form $6^2_2$.

\textbf{Case 1: $a_1, a_3$ even, $a_2$ odd}
If one of $a_1, a_3 > 2$, then we can form $N[\pm 4 \; \mp 1 \; \pm 2] \approx N[\pm 3 \; \pm 3] \implies F(L) \geq 6$. 

\textbf{Case 2: $a_1, a_3$ odd, $a_2$ even}
The first such link is $8^2_4 = N[3 \; 2 \; 3]$. $N[\pm1 \; \pm2 \; \pm3] = 6^2_2 \implies F(L) \geq 6$. 

\textbf{Case 3: $a_1, a_2, a_3$ even}
If one of $a_1, a_3 > 2$, we can form $N[\pm 4 \; \mp 2 \; \mp 2] = 6^2_2 \implies F(L) = 6$.

The upper bound $F(H) \leq 6$ on rational two-component links $H$ proves $F(L) = 6$.
\end{proof}

\begin{theorem}
All length 4 rational two-component links are locally fertile.
\end{theorem}

\begin{proof}
The only two-component links in $\mathcal{K}_4$ are $7^2_2 = N[3 \; 1 \; 1 \; 2]$, \\ $8^2_5 = N[3 \; 1 \; 2 \; 2]$, and $N[3 \; 2 \; 2 \; 3]$. All have the three six crossing two-component links as resultants. 

For $7^2_2$, $N[\pm 3 \; \pm 1 \; \mp 1 \; \pm 2] \approx 6^2_1$, $N[\pm 3 \; \pm 1 \; \pm 1 \; \mp 2] \approx 6^2_2$, and $N[\pm 3 \; \mp 1 \; \pm 1 \; \pm 2] \approx 6^2_3$.

For $8^2_5$, $N[\pm 3 \; \pm 1 \; \mp 2 \; \mp 2] \approx 6^2_1$, $N[\pm 3 \; \pm 1 \; 0 \; \pm 2] \approx 6^2_2$, and $N[\pm 1 \; \pm 1 \; \pm 2 \; \pm 2] \approx 6^2_3$.

For $N[3 \; 2 \; 2 \; 3]$, $N[\pm 3 \; 0 \; \pm 2 \; \pm 1] \approx 6^2_1$, $N[\pm 3 \; 0 \; 0 \; \pm 3] \approx 6^2_2$, and $N[\pm 3 \; \mp 2 \; \mp 2 \; \pm 1] \approx 6^2_3$.
\end{proof}

\begin{theorem}
All length 5 rational two-component links are locally fertile.
\end{theorem}

\begin{proof}
There are eight two-component links in $\mathcal{K}_5$. First consider \\ $N[3 \; 1 \; 1 \; 1 \; 3]$, $N[3 \; 1 \; 2 \; 1 \; 3]$, $N[3 \; 2 \; 2 \; 2 \; 3]$, $N[3 \; 2 \; 1 \; 1 \; 2]$, and $N[3 \; 2 \; 1 \; 2 \; 2]$. By Theorem \ref{f_across_L}, since each link has odd $a_1$ or $a_n$ and all length 4 rational links are locally fertile, these links are too.

Now consider $N[2 \; 2 \; 2 \; 2 \; 2]$. Choose $a_2 \rightarrow 0$ to form $N[4 \; 2 \; 2]$. This is a length three rational link with $a_2 > 1$ and one of $a_1, a_3 > 2$.

The remaining two-component links in $\mathcal{K}_5$ are $8^2_7 = [2 \; 1 \; 2 \; 1 \; 2]$ and $9^2_{11} = N[2 \; 2 \; 2 \; 1 \; 2]$. 

For $8^2_7$, choose $a_3 \rightarrow 0$ to form $N[2 \; 2 \; 2]$. Also, $N[\pm 2 \; \mp 1 \; \pm 2 \; \pm 1 \; \mp 2] \approx 6^2_1$ and $N[\pm 2 \; \pm 1 \; \mp 2 \; \pm 1 \; \mp 2] \approx 6^2_2$.

For $9^2_{11}$, choose $a_2 \rightarrow 0$ to form $N[4 \; 1 \; 2]$ and have $6^2_1$ and $6^2_2$ as resultants. Also, $N[2 \; 2 \; 2 \; 1 \; 0] \simeq 6^2_3$.
\end{proof}

From Theorem \ref{end_goal}, the last two results nearly complete our classification of rational two-component links. 

\begin{theorem} \label{link_max}
All rational two component links $L$ with $d(L) \geq 4$ are locally fertile. 
\end{theorem}

\begin{proof}
It remains to show $N[2 \; 1 \; 1 \; 1 \; 1 \; 2]$ is 6-fertile. \\ $N[\pm 2 \; \mp 1 \; \pm 1 \; \pm 1 \; \mp 1 \; \pm 2] \approx 6^2_1$, $N[\pm 2 \; \mp 1 \; \pm 1 \; \pm 1 \; \pm 1 \; \mp 2] \approx 6^2_2$, and $N[\pm 2 \; \mp 1 \; \pm 1 \; \mp 1 \; \pm 1 \; \pm 2] \approx 6^2_3$.
\end{proof}

\begin{corollary}
The only non-locally fertile two-component rational links are the torus links $T(p, 2)$, $p$ even, and length three links $N[2 \; a_2 \; 2]$ and $N[a_1 \; 1 \; a_3]$.
\end{corollary}

\section{Rational Knot Fertility Number} \label{knot_fert_num}

There are far more rational knots than links in trunks: $\mathcal{K}_6$ is composed of 12 links and 24 knots. Moreover, rational knots can have fertility numbers between 4 and 7 instead of the binary of 5 or 6 for rational links $L$ with $c(L) \geq 5$. To balance clarity with brevity, we will explicitly discuss length 3 and 4 rational knots to demonstrate the general approach, then provide the fertility number for all other higher length rational knots without explanation.

To tabulate $F(L)$, we wrote a program in Mathematica (available on Github \cite{gitRepository}) which identifies rational link resultants by their continued fraction. Its results matched those in Table 2 of \cite{MR3844207}, which were calculated for all rational knots through 10 crossings, and our results from \cite{MR4190429}, which were calculated by computing the Jones polynomial of all resultant diagrams. See Appendix 1 for all compiled results.

\subsection{Length 3 Knots}

\begin{theorem} \cite{MR3844207}
Let $K$ be a nontrivial prime knot. If $K = T(p, 2)$, $p$ odd, $F(K) = 3$. If $K = 4_1$, $F(K) = 4$. Otherwise, $F(K) \geq 4$ and $K$ has $5_2$ as a resultant.
\end{theorem}

\begin{theorem}
A rational knot $K$, $d(K) \geq 3$, has $F(K) \geq 5$ unless $d(K) = 3$, $a_2 = 1$, and $a_1, a_3$ are both odd.
\end{theorem}

\begin{proof}
\cite{MR3844207} similarly show nontrivial prime knots that are not pretzel knots $P(p_1, p_2, p_3)$ with $p_i$ all odd, positive, integers have $5_1$ as a resultant. Of the rational knots $K = N[a_1 \; a_2 \; ... \; a_n]$ with $d(K) \geq 3$, only rationals with $d(K) = 3$ and $a_2 = 1$ can be written as pretzels knots described by exactly three positive $p_i$. Further requirements on $a_i$ follow from the requirements from \cite{MR3844207}.
\end{proof}

We point the reader unfamiliar with pretzel links to \cite{pretzel}.

\begin{lemma}
Length three rational knots have $F(K) \leq 6$. 
\end{lemma}

\begin{proof}
$7_7 = N[\pm 2 \; \pm 1 \; \pm 1 \; \pm 1 \; \pm 2]$ can only be formed from an assignment $N[\pm a_1 \; \mp a_2 \; \mp a_3] \approx N[\pm (a_1 - 1) \; \pm 1 \; \pm (a_2 - 2) \; \pm 1 \; \pm (a_3 - 1)]$. Therefore $7_7$ is a resultant iff all $a_i$ are odd. 

When all $a_i$ are odd, if $c_i \geq 3$, $c(R) \geq 7 \implies 6_3$ is a resultant iff at least one of $a_i' \rightarrow 1'$. When $a_2 = 1$, $F(K) = 4$. The only remaining possibilities for six crossing resultants are $N[1' \; a_2' \; a_3'] \approx N[(a_2 + 1)' \; a_3']$ and similarly $N[a_1' \; (a_2 + 1)']$, which have resultants of length three or less, so $6_3$ cannot be a resultant.
\end{proof}

\begin{corollary} \label{rat3pret}
Let $K = N[a_1 \; a_2 \; a_3]$ be a nontrivial rational knot where $a_1, a_2, a_3$ are odd. If $a_2 = 1$, then $F(K) = 4$. Otherwise, \\$F(K) = 5$.
\end{corollary}

\begin{theorem}
Let $K = N[a_1 \; a_2 \; a_3]$ be a nontrivial rational knot where $a_1, a_2$ are odd and $a_3$ is even. If $a_2 = 1$, then $F(K) = 5$. Otherwise $F(K) = 6$.
\end{theorem}

\begin{proof}
$F(K) \geq 5$. If $a_2 = 1$, all resultants are of length three or below, so $6_3$ is not a resultant. Otherwise the following assignments are always possible: $N[\pm 1 \; \pm 3 \; \pm 2] \approx 6_1$, $N[\pm 3 \; \pm 1 \; \pm 2] \approx 6_2$, and $N[\pm 3 \; \mp 3 \; \mp 2] \approx 6_3 \implies F(K) \geq 6$.
\end{proof}

\begin{theorem}
Let $K = N[a_1 \; a_2 \; a_3]$ be a rational knot where $a_1$ is odd and $a_2, a_3$ are even. If $a_2 = 2$, then $F(K) = 5$. Otherwise $F(K) = 6$.
\end{theorem}

\begin{proof}
Trunk knot $7_5 = N[3 \; 2 \; 2]$ has at most two six crossing resultants, so $F(7_5) = 5$. It can form $N[\pm 3 \; \pm 2 \; \mp 2] \approx 6_2$ and $N[\pm 3 \; \mp 2 \; \pm 2] \approx 6_3$, but not $6_1 = N[4 \; 2]$.

When $a_2 \geq 4$, we have $N[\pm 1 \; \pm 4 \; \mp 2] \approx 6_1 \implies F(K) = 6$. However, branches with $a_2 = 2$ cannot form $6_1$. $2' \rightarrow 0$ produces length one knots, leaving $N[a_1' \; \pm 2 \; a_3']$ and $N[a_1' \; \mp 2 \; a_3']$. $a_1' \rightarrow c_1 \geq 3$ forces either $N[\pm c_1 \; \pm 2 \; a_3']$ (further resultants all length three or greater, i.e. not $6_1$) or $R_1 = N[\pm (c_1 - 1) \; \pm 1 \; \pm 1 \; a_3']$. 

\begin{itemize}
\item $a_3' \rightarrow \pm c_1 \geq 2 \implies d(R_1) = 4$, 
\item $a_3' \rightarrow 0 \implies d(R_1) = 1$, 
\item $a_3' \rightarrow \mp2 \implies d(R_1) = 2$, and
\item $a_3' \rightarrow \mp c_1 \leq 4 \implies d(R_1) = 3$. 
\end{itemize}

The only viable route to $6_1$ here is $N[\pm (c_1 - 1) \; \pm 1 \; \pm 1 \; \mp 2] \approx N[\pm (c_1 - 1) \; \pm 1 \; 0 \; \pm 1 \; \pm 1] \approx N[\pm (c_1 - 1) \; \pm 3] \not \approx 6_1$. 

The last unexamined cases are $N[\pm 1 \; \pm 2 \; a_3'] \approx N[\pm 3 \; a_3']$ and $N[\pm 1 \; \mp 2 \; a_3'] \approx N[\mp 1 \; a_3']$. The former case will have at least one tangle with an odd number of crossings since $a_3$ is even. The latter is length one for all $a_3$.
\end{proof}

\subsection{Length 4 Knots}

There are 7 knots in $\mathcal{K}_4$. We've already seen $6_3 = N[2 \; 1 \; 1 \; 2]$ is 5-fertile and $7_6 = N[2 \; 2 \; 1 \; 2]$ is 6-fertile. 

\begin{theorem}
$8_{11} = N[3 \; 2 \; 1 \; 2]$ and $8_9 = N[3 \; 1 \; 1 \; 3]$ are 5-fertile.
\end{theorem}

\begin{theorem}
$8_{12} = N[2 \; 2 \; 2 \; 2]$, $9_{13} = N[3 \; 2 \; 1 \; 3]$, and $9_{18} = N[3 \; 2 \; 2 \; 2]$ are 6-fertile.
\end{theorem}

\begin{proof}
See Appendix 1.
\end{proof}

\begin{theorem}
Knots $K = N[a_1 \; 1 \; 1 \; a_4]$ have $F(K) = 5$. 
\end{theorem}

\begin{proof}
$a_1, a_4$ are either both even or both odd. The eight forms of resultants are 

\begin{itemize}
\item $N[\pm c_1 \; \pm 1 \; \pm 1 \; \pm c_4]$
\item $N[\pm c_1 \; \mp 1 \; \pm 1 \; \pm c_4] \approx N[\pm (c_1 - 1) \; \pm 2 \; \pm c_4]$
\item $N[\pm c_1 \; \pm 1 \; \mp 1 \; \pm c_4] \approx N[\pm (c_1 + c_4 + 1)] $
\item $N[\pm c_1 \; \pm 1 \; \pm 1 \; \mp c_4] \approx N[\pm c_1 \; \pm 2 \; \pm (c_4 - 1)]$
\item $N[\pm c_1 \; \mp 1 \; \mp 1 \; \pm c_4] \approx N[\pm (c_1 + c_4 - 1)]$
\item $N[\pm c_1 \; \mp 1 \; \pm 1 \; \mp c_4] \approx N[\pm (c_1 - 1) \; \pm 1 \; \pm 1 \; \pm (c_4 - 1)]$
\item $N[\pm c_1 \; \pm 1 \; \mp 1 \; \mp c_4] \approx N[\pm (c_1 - c_4 + 1)]$
\item $N[\pm c_1 \; \mp 1 \; \mp 1 \; \mp c_4] \approx N[\pm (c_1 - c_4 - 1)]$
\end{itemize}

\noindent where $c_i = a_i - 2 m_i$, $m_i \in \mathbb{N}$. Obviously, for either parity, $6_3$ is a resultant. Also, $N[\pm 4 \; \mp 1 \; \pm 1 \; \mp 2]$ and $N[\pm 3 \; \pm 1 \; \pm 1 \; \pm 1] \approx 6_2$. Can we get $6_1$? In the length one cases and the first length four case, the answer is obviously no. 

The remaining cases are $N[\pm (c_1 - 1) \; \pm 2 \; \pm c_4]$ and $N[\pm (c_1 - 1) \; \pm 1 \; \pm 1 \; \pm (c_4 - 1)]$. Let $T = \pm (c_i - 1)$. If $c_i \geq 3$, $T$, and by extension $K$, is in canonical form. If $c_i = 2$, $T = \pm 1$ is absorbed, forming $N[\pm 3 \; \pm c_4]$ or $N[\pm 2 \; \pm 1 \; \pm (c_4 - 1)]$, where $c_4$ must be even. The only $N[2 \; n]$ resultants possible with this choice are $N[2 \; 2]$ and $N[2 \; 3]$. If $c_1 = 1$, $T = 0$ and each form collapses to a length one knot. Thus $6_1$ is not a resultant and these knots have $F(N[a_1 \; 1 \; 1 \; a_4]) = 5$.
\end{proof}

\begin{corollary}
The only length 4 knots $K$ with $F(K) = 5$ are branches of $N[3 \; 2 \; 1 \; 2]$ where $a_3 = 1$ and $N[a_1 \; 1 \; 1 \; a_4]$.
\end{corollary}

\begin{proof}
Note $F(N[3 \; 2 \; 3 \; 2]) = F(N[2 \; 3 \; 1 \; 2]) = F(N[3 \; 3 \; 1 \; 3]) = 6$.
\end{proof}

\subsection{Length 5 and Longer Knots}

First, length 5 knots. There are 12 knots in $\mathcal{K}_5$. 

\begin{theorem}
$8_{14} = N[2 \; 2 \; 1 \; 1 \; 2]$, $9_{23} = N[2 \; 2 \; 1 \; 2 \; 2]$, \newline
$8_{13} = N[3 \; 1 \; 1 \; 1 \; 2]$, $9_{20} = N[3 \; 1 \; 2 \; 1 \; 2]$, $9_{21} = N[3 \; 1 \; 1 \; 2 \; 2]$, $10_{29} = N[3 \; 1 \; 2 \; 2 \; 2]$, $11a337 = N[3 \; 2 \; 2 \; 1 \; 3]$, and $11a357 = N[3 \; 2 \; 1 \; 2 \; 3]$ are 6-fertile.
\end{theorem}

\begin{theorem}
$10_{25} = N[3 \; 2 \; 2 \; 1 \; 2]$, $10_{26} = N[3 \; 2 \; 1 \; 1 \; 3]$, and $11a236 = N[3 \; 2 \; 2 \; 2 \; 2]$ are 7-fertile. 
\end{theorem}

\begin{proof}
See Appendix 1.
\end{proof}

\begin{corollary}
$7_7 = N[2 \; 1 \; 1 \; 1 \; 2]$ is the only 5-fertile knot in $\mathcal{K}_5$.
\end{corollary}

\begin{corollary}
Knots $K = N[a_1 \; 1 \; a_3 \; 1 \; a_5]$, $a_1, a_5$ even and $a_3$ odd are the only length 5 knots with $F(K) = 5$. 
\end{corollary}

\begin{proof}
By the same methods as above. Note $F(N[2 \; 3 \; 1 \; 1 \; 2]) = 6$.
\end{proof}

\begin{theorem}
All knots $K \in \mathcal{K}_6$ have $F(K) > 5$.
\end{theorem}

\begin{theorem}
The sole knots in $\mathcal{K}_6$ that are only 6-fertile are $9_{26} = N[3 \; 1 \; 1 \; 1 \; 1 \; 2]$, $10_{27} = N[3 \; 2 \; 1 \; 1 \; 1 \; 2]$, $10_{30} = N[3 \; 1 \; 2 \; 1 \; 1 \; 2]$, $10_{32} = N[3 \; 1 \; 1 \; 1 \; 2 \; 2]$, $10_{33} = N[3 \; 1 \; 1 \; 1 \; 1 \; 3]$, $10_{43} = N[2 \; 1 \; 2 \; 2 \; 1 \; 2]$, $11a192 = N[3 \; 1 \; 2 \; 1 \; 2 \; 2]$, and $12a1039 = N[3 \; 1 \; 2 \; 2 \; 1 \; 3]$.
\end{theorem}


\begin{theorem}
The sole knots in $\mathcal{K}_7$ that are only 6-fertile are $9_{31} = N[2 \; 1^5 \; 2]$, $11a306 = N[3 \; 1^5 \; 3]$, and $N[3 \; 1 \; 2 \; 1 \; 2 \; 1 \; 3]$
\end{theorem}

\begin{theorem} \label{knot_max}
All knots of length 8 or higher are locally fertile.
\end{theorem}

\begin{proof}
All knots in $\mathcal{K}_8$ and $\mathcal{K}_9$ are 7-fertile, as is $N[2 \; 1^8 \; 2]$.
\end{proof}

We have now completely classified the fertility number of all rational knots. Interestingly, the stragglers at each length were the knots $N[2 \; 1^{\ell - 2} \; 2]$. Where other rational knots of the same length were approaching local fertility, these knots needed more component $1$ tangles to gain rational resultants. One could say they are the least rational rational knots, just as the golden ratio $\varphi = \frac{1+\sqrt{5}}{2} = (1 \; 1 \; 1 \; 1 \; 1 \; ...)$ has a superlatively slow rate of convergence in its approximation via continued fractions. Lo and behold, $N[2 \; 1^{\ell - 2} \; 2] \simeq N[1^\ell]$ with continued fraction $(2 \; 1^{\ell -1})$. 

\section{Knot Resultant Combinatorics} \label{combin}

\begin{note}
We may have a binomial coefficient ${n \choose k}$ where it is possible, in application, for $k > n$. In these cases, the binomial coefficient should be understood as equal to zero.
\end{note}

\begin{theorem}
A rational link $L = N[\pm a_1 \; \pm a_2 \; ... \; \pm a_n]$ has at most ${n \choose 1} + {n-1 \choose 2}$ unique resultants $R$ with $c(R) = c(L) - 2$ and at most ${n \choose 1} {n-1 \choose 1} + {n-1 \choose 3}$ unique resultants with $c(R) = c(L) - 3$.
\end{theorem}

\begin{corollary}
$L$ has at most $\frac{1}{2} (n^2 - n + 2)$ unique resultants $R$ with $c(R) = c(L) - 2$ and $\frac{1}{6} (n^3 + 5n- 6)$ unique resultants with $c(R) = c(L) - 3$.
\end{corollary}

\begin{theorem} \label{rat_count}
A rational shadow $L = N[a_1' \; a_2' \; ... \; a_n']$ has at most $\lceil \frac{a_1+1}{2} \rceil \prod_{i=2}^n (a_i + 1)$ unique resultants.
\end{theorem}

\begin{proof}
The integer tangle $x'$ has resolutions with $x - 2j$ crossings, $0 \leq j \leq x \implies$ there are $x+1$ resolutions of each integer tangle. Our possible resolutions are nearly halved for the first tangle, as it sets the sign convention for the link, i.e. $N[\pm c_1 \; a_2' \; ... \; a_n']$ is indistinguishable from $N[\mp c_1 \; a_2' \; ... \; a_n']$ as we can freely redefine crossing parity here before assigning any other crossings. 
\end{proof}

For any link shadow $L$, all we know it has at most $2^{c(L)-1}$ unique resultants when we count a knot and its mirror image the same. This theorem gives a significantly sharper upper bound for rational links, yet there is a long way to go. The knot $8_{12} = N[2 \; 2 \; 2 \; 2]$ has 11 unique resultants, far fewer than $2 \times 3 \times 3 \times 3 = 54$, yet still much sharper than $2^7 = 128$. Theorem \ref{palTheorem} is likely useful for improvements.

We will now explicitly enumerate the number of each link type $L'$ produced by rational links $L$ of lengths 1 and 2. This can be converted to a probability distribution by dividing by the total number of resultants $2^{c(L)}$. 

\begin{theorem}
$D[a_1' \; a_2' \; ... \; a_n'] \rightarrow 2^{a_n} N[a_1' \; a_2' \; ... \; a_{n-1}']$, $n \geq 2$. $D[a_1'] \rightarrow 2^{a_1}$ unknots.
\end{theorem}

This notation means we get $2^{a_n}$ diagrams described by $N[a_1' \; a_2' \; ... \; a_{n-1}']$.

\begin{proof}
A shadow RI move is an invariant on knot shadows by creating two copies (one for each resolution) of the shadow prior to its application to a straight strand. Applying the denominator to a rational tangle connects the right ends of the $a_n$ tangle, allowing it to be untwisted $a_n$ times. For $n \geq 2$, what remains is the rotated version of the numerator closure. For $n = 1$, the complicating twists must have been done on an unknot. 
\end{proof}

\begin{corollary}
$N[a_1' \; a_2' \; ... \; a_n' \; 0] \rightarrow 2^{a_n} N[a_1' \; a_2' \; ... \; a_{n-1}']$, $n \geq 2$. $N[a_1' \; 0] \rightarrow 2^{a_1}$ unknots. 
\end{corollary}

\begin{theorem}
An $a_1'$ $(N[a_1'])$ shadow has ${{a_1} \choose {\frac{a_1 - k}{2}}}$ $+k$ $(N[+k])$ and $-k$ $(N[-k])$ resultants, where $a_1$ and $k$ have the same parity.
\end{theorem}

Interestingly, links lack many unlink resultants. Unknots are the most likely resultant from $N[a_1']$ knot shadows. Trivial resultants are fairly dominant in general, especially at low crossing numbers. \cite{MR3084750} found experimentally all knot shadows with twelve or fewer crossings had a higher unknot probability than any nontrivial resultant probability (they did show the 14-crossing torus knot T(7,3) was a counterexample). Furthermore, we are unaware of any knot shadow where the resultant probability of any nontrivial knot is greater than or equal to 50\%. 

$N[2']$, the first nontrivial link shadow, has four total resultants: two 2-component unlinks and two Hopf links. The $N[a_1']$ link has more Hopf links than unlinks for all $a_1 \geq 4$. In fact, with a long enough initial link shadow $N[a_1']$, there will be more particular nontrivial link resultants $N[c_1]$ than unlinks. 

\begin{lemma} For all $k$, there is some $n$ such that
\begin{equation*}
2 {{n} \choose {\frac{n-k}{2}}} \geq {{n} \choose {\frac{n}{2}}}.
\end{equation*}
$n,k$ are even integers, $n \geq k$.
\end{lemma}

\begin{proof}
Converting to the factorial form of the binomial coefficient, the above inequality holds iff

\begin{align*}
2 \left[ \left( \frac{n}{2} \right) ! \right]^2 &\geq \left(\frac{n}{2} - \frac{k}{2} \right)! \left(\frac{n}{2} + \frac{k}{2} \right)! \iff \\
2 \left( \frac{n}{2} \right) \left( \frac{n}{2} - 1 \right) ... \left( \frac{n}{2} - \frac{k}{2} + 1 \right) &\geq \left( \frac{n}{2} + \frac{k}{2} \right) \left( \frac{n}{2} + \frac{k}{2} - 1\right) ... \left( \frac{n}{2} + 1 \right) 
\end{align*}

The left polynomial has a leading coefficient two times larger than that of the right and will eventually surpass the right polynomial for large $n$.
\end{proof}

\begin{corollary}
For any link $N[c_1]$, there is some $A$ such that $N[a_1']$, $a_1 > A$, has more $N[c_1]$ than unlink resultants. 
\end{corollary}

For $k=2$, or Hopf link resultants, the above inequality is simply $n \geq \frac{n}{2} + 1$, which is true for all integers greater than 1. For $k=4$, or $4^2_1$ resultants, the inequality reduces to $n (\frac{n}{2} - 1) \geq (\frac{n}{2} + 1) (\frac{n}{2} + 2)$, or $n \geq 5 + \sqrt{33} \approx 10.745$. Indeed, we see $N[10']$ has 252 unlink resultants and 240 $4^2_1$ resultants, while $N[12']$ has 924 unlink and 990 $4^2_1$ resultants. 

\begin{lemma} \label{sum_approach}
The shadow of a rational tangle 
\begin{equation}
N[a_1' \; a_2' \; ... \; a_n'] \rightarrow \sum^{a_1}_{m=-a_1} {{a_1} \choose {\frac{a_1 - |m|}{2}}} N[ m \; a_2' \; ... \; a_n'],
\end{equation}
where $m$ is restricted to the parity of $a_n$. 
\end{lemma}

\begin{proof}
This follows directly from the resolution of integer tangles.
\end{proof}

\begin{theorem}
Let $a_1, a_2, k, \ell$ be even, positive integers, $2 \leq k \leq a_1$, $2 \leq \ell \leq a_2$. Then the $N[a_1' \; a_2']$ knot shadow has 

\begin{equation*}
{{a_1} \choose {\frac{a_1 - k}{2}}} {{a_2} \choose {\frac{a_2 - \ell}{2}}} + {{a_1} \choose {\frac{a_1 - \ell}{2}}} {{a_2} \choose {\frac{a_2 - k}{2}}} 
\end{equation*}

\noindent $N[k \; \ell]$ and $N[ (k-1) \; 1 \; (\ell-1)]$ resultants and 

\begin{equation*}
2^{a_2} {{a_1} \choose {\frac{a_1}{2}}} + \left[ 2^{a_1} - {{a_1} \choose {\frac{a_1}{2}}} \right] {{a_2} \choose {\frac{a_2}{2}}} 
\end{equation*}

\noindent unknot resultants.

\end{theorem}

\begin{proof}
Using Lemma \ref{sum_approach}, we start with ${{a_1} \choose {\frac{a_1}{2}}}$  $N[0 \; a_2']$ diagrams, which are all unknots. We can also automatically untangle the knot shadow if we take $a_2 \rightarrow \ell = 0$, but we've already counted the cases where $k = 0$. New unknots come from assignments where $k$ is non-trivial. There are exactly $2^{a_1} - {{a_1} \choose {\frac{a_1}{2}}}$ of these.

What remains are $\sum_m {{a_1} \choose {\frac{a_1 - |m|}{2}}} \sum_{n} {{a_2} \choose {\frac{a_2 - |n|}{2}}} N[ m \; n ]$ diagrams, where $m$, $n \not= 0$. By symmetry, for specific positive indices $k$, $\ell$, we have the proposed number of knot diagrams $N[ \pm k \; \pm \ell]$ and $N[ \pm k \; \mp \ell] \approx N[\pm (k-1) \; \pm 1 \; \pm (\ell -1)]$. 
\end{proof}

The combinatorics of the odd-odd length two rational links is a natural extension.

\begin{theorem}
Let $a_1, k$ be even integers and $a_2, \ell$ be odd integers, $2 \leq k \leq a_1$, $2 \leq \ell \leq a_2$. Then the $N[a_1' \; a_2']$ knot shadow has 

\begin{equation*}
2 {{a_1} \choose {\frac{a_1 - k}{2}}} {{a_2} \choose {\frac{a_2 - \ell}{2}}}
\end{equation*}

\noindent $N[k \; \ell]$ and $N[(k-1) \; 1 \; (\ell-1)]$ resultants,

\begin{equation*}
2 {{a_2} \choose {\frac{a_2 - 1}{2}}} \left[ {{a_1} \choose {\frac{a_1 - k}{2}}} + {{a_1} \choose {\frac{a_1 - (k+2) }{2}}}  \right]
\end{equation*}

\noindent $N[k+1]$ resultants, and

\begin{equation*}
2^{a_2} {{a_1} \choose {\frac{a_1}{2}}} + 2 {{a_1} \choose {\frac{a_1-2}{2}}} {{a_2} \choose {\frac{a_2-1}{2}}} 
\end{equation*}

\noindent unknot resultants.
\end{theorem}

\begin{proof}
Without loss of generality, choose $a_1$ to be even. Like above, we can take $a_1 \rightarrow 0$, making $2^{a_2} {{a_1} \choose {\frac{a_1}{2}}}$ unknot diagrams. 

By Lemma \ref{sum_approach}, if $a_1 \rightarrow \pm k$, $k \geq 0$, then for all $\ell \leq a_2$, there are ${{a_1} \choose {\frac{a_1 - k}{2}}} {{a_2} \choose {\frac{a_2 - \ell}{2}}}$ $N[\pm k \; \pm \ell]$ diagrams and ${{a_1} \choose {\frac{a_1 - k}{2}}} {{a_2} \choose {\frac{a_2 - \ell}{2}}}$ $N[\pm k \; \mp \ell] \approx N[\mp (k-1) \; \mp 1 \; \mp (\ell-1)]$ diagrams. For $k,\ell \geq 3$, these diagrams are in canonical form, hence reduced, and unique from each other. 
 
This leaves $N[\pm k \; \pm1]\approx N[\pm(k+1)]$ and $N[\pm k \; \mp 1 ]\approx N[\pm(k - 1)]$. $N[\pm 2 \; \mp 1] \approx N[\pm 1 \; \pm 1 \; 0] \approx D[\pm 2],$ or the unknot. 
\end{proof}

\begin{theorem} \label{rat_gen_theorem} \cite{MR3084750}
The $N[a_1' \; a_2' \; ... \; a_n']$ shadow has $\prod^n_{i=1} {{a_i} \choose {\frac{a_i - c_i}{2}}}$ resultants of the form $N[c_1 \; c_2 \; ... \; c_n]$.
\end{theorem}

This implicitly details all resultants of a rational link, and guarantees that they will all be rational. It would be nice to be more specific about what particular rational links occur, but this rapidly becomes a grueling accounting exercise. For example, using Lemma \ref{sum_approach}, one can show the $N[a_1' \; a_2' \; a_3']$ shadow produces 

\begin{align*}
2^{a_2 + 1} {{a_1} \choose {\frac{a_1}{2}}} {{a_3} \choose {\frac{a_3 - 1}{2}}} &+ 2 {{a_1} \choose {\frac{a_1 -2}{2}}} {{a_2} \choose {\frac{a_2 - 2}{2}}} {{a_3} \choose {\frac{a_3 - 1}{2}}}  \\
&+ 2 {{a_2} \choose {\frac{a_2}{2}}} \sum_{c_1 > 0} {{a_1} \choose {\frac{a_1-c_1}{2}}} \left[ {{a_3} \choose {\frac{a_3 - c_1 + 1}{2}}}+ {{a_3} \choose {\frac{a_3 - c_1 - 1}{2}}}  \right] 
\end{align*}

\noindent unknots if $a_3$ is odd but $a_1$ and $a_2$ and even. However, if $a_2$ and $a_3$ are odd and only $a_1$ is even, there are instead 

\begin{align*}
2^{a_2 + 1} {{a_1} \choose {\frac{a_1}{2}}} {{a_3} \choose {\frac{a_3 - 1}{2}}} &+ 2 {{a_1} \choose {\frac{a_1 -2}{2}}} {{a_2} \choose {\frac{a_2 - 1}{2}}} {{a_3} \choose {\frac{a_3 - 3}{2}}}  \\
&+ 2 {{a_2} \choose {\frac{a_2 - 1}{2}}} {{a_3} \choose {\frac{a_3 -1}{2}}} \left[ 2^{a_1} -  {{a_1} \choose {\frac{a_1}{2}}}  \right] 
\end{align*}

\noindent unknots. For length three rational links alone, there still remains the nontrivial resultants for these cases, the remaining case where all three $a_i$ are odd, and the three cases where the rational tangle is a true link.

\section{Acknowledgements}

Thank you to Dr. Emily Peters for her helpful suggestions on improving this text.

\newpage

\section{Appendix I: Rational Link Fertility Numbers}
	
\begin{table}[h]
\caption{Fertility numbers of rational knots through 10 crossings.}
\begin{tabular}{l|l|l|l|l|l|l|l|l}
K        & CF        & F & K         & CF            & F & K         	& CF              & F \\ \hline
$3_1$    & 3         & 3 	& $9_8$    & 2 4 1 2   & 6  	& $10_{17}$ & 4 1 1 4       & 5  \\
$4_1$    & 2 2       & 4 	& $9_9$    & 4 2 3     & 5		& $10_{18}$ & 4 1 1 2 2     & 6  \\
$5_1$    & 5         & 3 	& $9_{10}$ & 3 3 3     & 5		& $10_{19}$ & 4 1 1 1 3     & 6  \\
$5_2$    & 3 2       & 4 	& $9_{11}$ & 4 1 2 2   & 6		& $10_{20}$ & 3 5 2         & 6  \\
$6_1$    & 4 2       & 4 	& $9_{12}$ & 4 2 1 2   & 6 	& $10_{21}$ & 3 4 1 2       & 5  \\
$6_2$    & 3 1 2     & 5 	& $9_{13}$  & 3 2 1 3       & 6	& $10_{22}$ & 3 3 1 3       & 6  \\
$6_3$    & 2 1 1 2   & 5 	& $9_{14}$  & 4 1 1 1 2     & 5	& $10_{23}$ & 3 3 1 1 2     & 7  \\
$7_1$    & 7         & 3 	& $9_{15}$  & 2 3 2 2       & 6	& $10_{24}$ & 3 2 3 2       & 6 \\
$7_2$    & 5 2       & 4 	& $9_{17}$  & 2 1 3 1 2     & 5	& $10_{25}$ & 3 2 2 1 2     & 7  \\
$7_3$    & 4 3       & 5 	& $9_{18}$  & 3 2 2 2       & 6	& $10_{26}$ & 3 2 1 1 3     & 7  \\
$7_4$    & 3 1 3     & 4 	& $9_{19}$  & 2 3 1 1 2     & 6	& $10_{27}$ & 3 2 1 1 1 2     & 6 \\
$7_5$    & 3 2 2     & 5 	& $9_{20}$  & 3 1 2 1 2     & 6	& $10_{28}$ & 3 1 3 1 2       & 6 \\
$7_6$    & 2 2 1 2   & 6 	& $9_{21}$  & 3 1 1 2 2     & 6	& $10_{29}$ & 3 1 2 2 2       & 6 \\
$7_7$    & 2 1 1 1 2 & 5 	& $9_{23}$  & 2 2 1 2 2     & 6	& $10_{30}$ & 3 1 2 1 1 2     & 6 \\
$8_1$    & 6 2       & 4 	& $9_{26}$  & 3 1 1 1 1 2   & 6	& $10_{31}$ & 3 1 1 3 2       & 6 \\
$8_2$    & 5 1 2     & 5 	& $9_{27}$  & 2 1 2 1 1 2   & 7	& $10_{32}$ & 3 1 1 1 2 2     & 6 \\	
$8_3$    & 4 4       & 5 	& $9_{31}$  & 2 1 1 1 1 1 2 & 6	& $10_{33}$ & 3 1 1 1 1 3     & 6 \\	
$8_4$    & 4 1 3     & 5 	& $10_1$    & 8 2           & 4	& $10_{34}$ & 2 5 1 2         & 6 \\
$8_6$    & 3 3 2     & 6 	& $10_2$    & 7 1 2         & 5	& $10_{35}$ & 2 4 2 2         & 6 \\
$8_7$    & 4 1 1 2   & 5 	& $10_3$    & 6 4           & 4	& $10_{36}$ & 2 4 1 1 2       & 6 \\
$8_8$    & 2 3 1 2   & 6	& $10_4$    & 6 1 3         & 5	& $10_{37}$ & 2 3 3 2         & 6 \\
$8_9$    & 3 1 13    & 5	& $10_5$    & 6 1 1 2       & 5	& $10_{38}$ & 2 3 1 2 2       & 6 \\
$8_{11}$ & 3 2 1 2   & 5	& $10_6$    & 5 3 2         & 6	& $10_{39}$ & 2 2 3 1 2       & 7 \\
$8_{12}$ & 2 2 2 2   & 6	& $10_7$    & 5 2 1 2       & 5	& $10_{40}$ & 2 2 2 1 1 2     & 7 \\
$8_{13}$ & 3 1 1 1 2 & 6	& $10_8$    & 5 1 4.         &  5	& $10_{41}$ & 2 2 1 2 1 2     & 7 \\
$8_{14}$ & 2 2 1 1 2 & 6	& $10_9$    & 5 1 1 3       & 5 	& $10_{42}$ & 2 2 1 1 1 1 2   & 7 \\
$9_1$    & 9         & 3 	& $10_{10}$ & 5 1 1 1 2     & 6	& $10_{43}$ & 2 1 2 2 1 2     & 6 \\
$9_2$    & 7 2       & 4 	&  $10_{11}$ & 4 3 3         & 6	& $10_{44}$ & 2 1 2 1 1 1 2   & 7 \\
$9_3$    & 6 3       & 5 	& $10_{12}$ & 4 3 1 2       & 6	& $10_{45}$ & 2 1 1 1 1 1 1 2 & 7 \\
$9_4$    & 5 4       & 4 	& $10_{13}$ & 4 2 2 2       & 6 \\
$9_5$    & 5 1 3 & 4		& $10_{14}$ & 4 2 1 1 2     & 6 \\
$9_6$    & 5 2 2     & 5	& $10_{15}$ & 4 1 3 2       & 6 \\
$9_7$    & 3 4 2     & 6	& $10_{16}$ & 4 1 2 3       & 5 \\

\end{tabular}
\end{table}

\begin{table}[h]
\caption{Fertility number of rational 2-component links through 9 crossings.}
\begin{tabular}{l|l|l|l|l|l}
K        & CF        & F & K         & CF            & F  \\ \hline
$2^2_1$    & 2         & 2  & $8^2_7$    & 2 1 2 1 2     & 6 	 \\
$4^2_1$    & 4       & 4 	& $8^2_8$    & 2 1 1 1 1 2       & 6 \\
$5^2_1$    & 2 1 2         & 3 & $9^2_1$    & 6 1 2     & 5 	 \\	
$6^2_1$    & 6       & 4 	& $9^2_2$    & 5 1 1 2     & 6 \\
$6^2_2$    & 3 3       & 5 	& $9^2_3$    & 4 3 2   & 6  \\
$6^2_3$    & 2 2 2     & 5 	& $9^2_4$    & 4 1 4     & 5 \\
$7^2_1$    & 4 1 2   & 5 	& $9^2_5$ & 4 1 1 3  &  6  \\
$7^2_2$    & 3 1 1 2         & 6 	& $9^2_6$ & 3 3 1 2   & 6	  \\
$7^2_3$    & 2 3 2       & 5 & $9^2_7$ & 3 2 2 1 2   & 6  \\	
$8^2_1$    & 8       & 4 	& $9^2_8$  & 3 1 3 2     & 6	\\
$8^2_2$    & 5 3     & 5 	& $9^2_9$  & 3 1 1 1 3     & 6	  \\
$8^2_3$    & 4 2 2     & 6 & $9^2_{10}$  & 2 5 2       & 5	\\	
$8^2_4$    & 3 2 3   & 6   & $9^2_{11}$  & 2 2 2 1 2     & 6  \\
$8^2_5$    & 3 1 2 2 & 6  & $9^2_{12}$  & 2 2 1 1 1 2       & 6	 \\
$8^2_6$    & 2 4 2       & 5 

\end{tabular}
\end{table}

\begin{table}[]
\caption{Fertility numbers of rational 2-component links with 10 crossings. Ordering follows \cite{MR0258014}.}
\begin{tabular}{l|l|l|l}
CF          & F & CF            & F \\ \hline
10          & 2 & 3 2 1 2 2     & 6 \\
7 3         & 5 & 3 1 4 2       & 6 \\
6 2 2       & 6 & 3 1 2 1 3     & 6 \\
5 5         & 5 & 3 1 1 2 1 2   & 6 \\
5 2 3       & 6 & 3 1 1 1 1 1 2 & 6 \\
5 1 2 2     & 6 & 2 6 2         & 5 \\
4 4 2       & 6 & 2 3 2 1 2     & 6 \\
4 2 4       & 6 & 2 3 1 1 1 2   & 6 \\
4 2 1 3     & 6 & 2 2 2 2 2     & 6 \\
4 1 2 1 2   & 6 & 2 2 1 1 2 2   & 6 \\
4 1 1 1 1 2 & 6 & 2 1 4 1 2     & 6 \\
3 4 3       & 6 & 2 1 3 1 1 2   & 6 \\
3 3 2 2     & 6 & 2 1 1 2 1 1 2 & 6 \\
3 2 2 3     & 6        
\end{tabular}
\end{table}

\newpage 

\textit{}

In the following tables, a number $m$ is fixed, but a number $n^*$ can take any value $n + 2k$, $k \in \mathbb{N}$ and the knot will retain its fertility number.

\begin{table}[h]
\begin{tabular}{c|c|c|c}
F(L) & Length 1 & Length 2 & Length 3 \\ \hline
2	&	2	&	&	 \\ \hline
4	&	$4^*$	&	&	\\ \hline
5	&		&	$3^*$ $3^*$	& \begin{tabular}[c]{@{}c@{}}$2$ $3^*$ $2$ \\ $4^*$ $1$ $2^*$\end{tabular} \\ \hline
6	&	&	&	$4^*$ $3^*$ $2^*$
\end{tabular}
\caption{Length 1, 2, and 3 rational link fertility number. Links of longer length have fertility number 6.}
\end{table}

\begin{table}[h]
\begin{tabular}{c|c|c}
F(L) & Length 1 & Length 2 \\ \hline
3	&	$3^*$	&	 \\ \hline
4	&		&	\begin{tabular}[c]{@{}c@{}}$2^*$ $2^*$ \\ $2$ $3^*$ \end{tabular} \\ \hline
5	&	&	$4^*$ $3^*$
\end{tabular}
\caption{Length 1 and 2 rational knot fertility number.}
\end{table}

\begin{table}[h]
\begin{tabular}{c|c|c|c}
F(K) & 3 1 2             & 3 1 3             & 3 2 2             \\ \hline
4    &                   & $3^*$ 1 $3^*$     &                   \\
5    & $3^*$ 1 $2^*$     & $3^*$ $3^*$ $3^*$ & $3^*$ 2 $2^*$     \\
6    & $3^*$ $3^*$ $2^*$ &                   & $3^*$ $2^*$ $2^*$
\end{tabular}
\caption{Length 3 rational knot fertility number.}
\end{table}

\begin{landscape}

\begin{table}[]
\begin{tabular}{c|c|c|c|c|c|c|c}
F(K) & 2 1 1 2                                                                               & 2 2 1 2                                                                                                 & 3 1 1 3                                                                               & 3 2 1 2                 & 2 2 2 2                 & 3 2 1 2                 & 3 2 2 2                                                                                                     \\ \hline
5    & $2^*$ 1 1 $2^*$                                                                       &                                                                                                         & $3^*$ 1 1 $3^*$                                                                       & $3^*$ $2^*$ 1 $2^*$     &                         &                         &                                                                                                             \\ \hline
6    & \begin{tabular}[c]{@{}c@{}}$2^*$ $3^*$ 1 $2^*$\\ $2^*$ $3^*$ $3^*$ $2^*$\end{tabular} & \begin{tabular}[c]{@{}c@{}}$2^*$ $2^*$ $1^*$ 2\\ $2^*$ $2^*$ 1 $2^*$\\ 2 $2^*$ $1^*$ $2^*$\end{tabular} & \begin{tabular}[c]{@{}c@{}}$3^*$ $3^*$ 1 $3^*$\\ $3^*$ $3^*$ $3^*$ $3^*$\end{tabular} & $3^*$ 2 $3^*$ $2^*$     & $2^*$ $2^*$ $2^*$ $2^*$ & $3^*$ $2^*$ 1 $2^*$     & $3^*$ 2 $2^*$ 2                                                                                             \\ \hline
7    &                                                                                       & $4^*$ $2^*$ $3^*$ $4^*$                                                                                 &                                                                                       & $3^*$ $4^*$ $3^*$ $2^*$ &                         & $3^*$ $2^*$ $3^*$ $2^*$ & \begin{tabular}[c]{@{}c@{}}$3^*$ $4^*$ $2^*$ 2\\ $3^*$ 2 $2^*$ $4^*$\\ $3^*$ $4^*$ $2^*$ $4^*$\end{tabular}
\end{tabular}
\caption{Fertility number for non-locally-fertile length 4 rational knots.}
\end{table}

\begin{table}[]
\begin{tabular}{c|c|c|c|c|c|c}
F(K) & 2 1 1 1 2                     & 3 1 1 1 2                                                                                     & 2 2 1 1 2                                                                                         & 2 2 1 2 2                     & 3 1 2 1 2                                                                                     & 3 1 1 2 2                     \\ \hline
5    & $2^*$ 1 $1^*$ 1 $2^*$         &                                                                                               &                                                                                                   &                               &                                                                                               &                               \\ \hline
6    & $2^*$ $3^*$ 1 $1^*$ $2^*$     & \begin{tabular}[c]{@{}c@{}}$3^*$ 1 $1^*$ $1^*$ 2\\ $3^*$ 1 $1^*$ 1 $2^*$\end{tabular}         & \begin{tabular}[c]{@{}c@{}}2 $2^*$ 1 $1^*$ $2^*$  \\ $2^*$ $2^*$ 1 $1^*$ 2\end{tabular}           & $2^*$ $2^*$ 1 $2^*$ $2^*$     & \begin{tabular}[c]{@{}c@{}}$3^*$ $1^*$ $2^*$ 1 $2^*$\\ $3^*$ 1 $2^*$ $1^*$ $2^*$\end{tabular} & $3^*$ 1 $1^*$ $2^*$ $2^*$     \\ \hline
7    & $2^*$ $3^*$ $3^*$ $1^*$ $2^*$ & \begin{tabular}[c]{@{}c@{}}$3^*$ 1 1 $3^*$ $4^*$\\ $3^*$ $3^*$ $1^*$ $1^*$ $2^*$\end{tabular} & \begin{tabular}[c]{@{}c@{}}$4^*$ 2 $1^*$ $1^*$ $4^*$\\ $2^*$ $2^*$ $3^*$ $1^*$ $2^*$\end{tabular} & $2^*$ $2^*$ $3^*$ $2^*$ $2^*$ & $3^*$ $3^*$ $2^*$ $3^*$ $2^*$                                                                 & $3^*$ $3^*$ $1^*$ $2^*$ $2^*$
\end{tabular}
\caption{Fertility number for non-locally-fertile length 5 rational knots with 9 or fewer crossings.}
\end{table}

\begin{table}[]
\begin{tabular}{c|c|c|c}
F(K) & 3 1 2 2 2                                                                                         & 3 2 2 1 3                                                                                             & 3 2 1 2 3                     \\ \hline
6    & $3^*$ 1 $2^*$ $2^*$ 2                                                                             & $3^*$ 2 $2^*$ 1 $3^*$                                                                                 & $3^*$ $2^*$ 1 $2^*$ $3^*$     \\ \hline
7    & \begin{tabular}[c]{@{}c@{}}$3^*$ 1 $2^*$ $2^*$ $4^*$\\ $3^*$ $3^*$ $2^*$ $2^*$ $2^*$\end{tabular} & \begin{tabular}[c]{@{}c@{}}$3^*$ $4^*$ $2^*$ $1^*$ $3^*$\\ $3^*$ $2^*$ $2^*$ $3^*$ $3^*$\end{tabular} & $3^*$ $2^*$ $3^*$ $2^*$ $3^*$
\end{tabular}
\caption{Fertility number for non-locally fertile length 5 rational knots with 10 or more crossings.}
\end{table}

\begin{table}[]
\begin{tabular}{c|c|c|c|c}
F(K) & 3 1 1 1 1 2                                                                                                                     & 3 2 1 1 1 2                                                                                               & 3 1 2 1 1 2                                                                                               & 3 1 1 1 2 2                                                                                               \\ \hline
6    & $3^*$ 1 $1^*$ 1 $1^*$ 2                                                                                                         & $3^*$ $2^*$ 1 $1^*$ 1 $2^*$                                                                               & $3^*$ 1 $2^*$ 1 $1^*$ $2^*$                                                                               & $3^*$ 1 $3^*$ 1 $2^*$ $2^*$                                                                               \\ \hline
7    & \begin{tabular}[c]{@{}c@{}}$3^*$ $3^*$ $1^*$ 1 $1^*$ 2\\ $3^*$ 1 $1^*$ $3^*$ $1^*$ 2\\ $3^*$ 1 $1^*$ 1 $1^*$ $4^*$\\ $3^*$ $3^*$ $1^*$ $3^*$ $1^*$ $4^*$ \end{tabular} & \begin{tabular}[c]{@{}c@{}}$3^*$ $2^*$ $3^*$ $1^*$ 1 $2^*$\\ $3^*$ $2^*$ 1 $1^*$ $3^*$ $2^*$\\ $3^*$ $2^*$ $3^*$ $1^*$ $3^*$ $4^*$ \end{tabular} & \begin{tabular}[c]{@{}c@{}}$3^*$ $3^*$ $2^*$ 1 $1^*$ $2^*$\\ $3^*$ 1 $2^*$ $3^*$ $1^*$ $2^*$ \\ $3^*$ $3^*$ $2^*$ $3^*$ $1^*$ $2^*$ \end{tabular} & \begin{tabular}[c]{@{}c@{}}$3^*$ $3^*$ $3^*$ 1 $2^*$ $2^*$\\ $3^*$ 1 $3^*$ $3^*$ $2^*$ $2^*$\\ $3^*$ $3^*$ $3^*$ $3^*$ $2^*$ $2^*$\end{tabular}
\end{tabular}
\caption{Fertility number for some non-locally-fertile length 6 rational knots.}
\end{table}

\begin{table}[]
\begin{tabular}{c|c|c|c|c}
F(K) & 3 1 1 1 1 3                                                                                                   & 2 1 2 2 1 2                                                                                                   & 3 1 2 1 2 2                                                                                                                                                           & 3 1 2 2 1 3                                                                                                   \\ \hline
6    & $3^*$ 1 $1^*$ $1^*$ 1 $3^*$                                                                                   & $2^*$ 1 $2^*$ $2^*$ 1 $2^*$                                                                                   & $3^*$ 1 $2^*$ 1 $2^*$ 2                                                                                                                                               & $3^*$ 1 $2^*$ $2^*$ 1 $3^*$                                                                                   \\ \hline
7    & \begin{tabular}[c]{@{}c@{}}$3^*$ $3^*$ $1^*$ $1^*$ 1 $3^*$\\ $3^*$ $3^*$ $1^*$ $1^*$ $3^*$ $3^*$\end{tabular} & \begin{tabular}[c]{@{}c@{}}$2^*$ $3^*$ $2^*$ $2^*$ 1 $2^*$\\ $2^*$ $3^*$ $2^*$ $2^*$ $3^*$ $2^*$\end{tabular} & \begin{tabular}[c]{@{}c@{}}$3^*$ $3^*$ $2^*$ 1 $2^*$ 2\\ $3^*$ 1 $2^*$ $3^*$ $2^*$ 2\\ $3^*$ 1 $2^*$ 1 $2^*$ $4^*$\\ $3^*$ $3^*$ $2^*$ $3^*$ $2^*$ $4^*$\end{tabular} & \begin{tabular}[c]{@{}c@{}}$3^*$ $3^*$ $2^*$ $2^*$ 1 $3^*$\\ $3^*$ $3^*$ $2^*$ $2^*$ $3^*$ $3^*$\end{tabular}
\end{tabular}
\caption{Fertility number for the remaining non-locally-fertile length 6 rational knots.}
\end{table}

\begin{table}[]
\begin{tabular}{c|c|c|c}
F(K) & 2 1 1 1 1 1 2                                                                                                                 & 3 1 1 1 1 1 3                                                                                                                 & 3 1 2 1 2 1 3                                                                                                                 \\ \hline
6    & $2^*$ 1 $1^*$ 1 $1^*$ 1 $2^*$                                                                                                 & $3^*$ 1 $1^*$ $3^*$ $1^*$ 1 $3^*$                                                                                             & $3^*$ 1 $2^*$ 1 $2^*$ 1 $3^*$                                                                                                 \\ \hline
7    & \begin{tabular}[c]{@{}c@{}}$2^*$ $1^*$ $1^*$ $3^*$ $1^*$ $1^*$ $2^*$\\ $2^*$ $3^*$ $1^*$ $1^*$ $1^*$ $1^*$ $2^*$\end{tabular} & \begin{tabular}[c]{@{}c@{}}$3^*$ $1^*$ $1^*$ $3^*$ $1^*$ $1^*$ $3^*$\\ $3^*$ $3^*$ $1^*$ $1^*$ $1^*$ $1^*$ $3^*$\end{tabular} & \begin{tabular}[c]{@{}c@{}}$3^*$ $1^*$ $2^*$ $3^*$ $2^*$ $1^*$ $3^*$\\ $3^*$ $3^*$ $2^*$ $1^*$ $2^*$ $1^*$ $3^*$\end{tabular}
\end{tabular}
\caption{Fertility number for non-locally-fertile length 7 rational knots.}
\end{table}

\end{landscape}

\end{document}